\documentclass{article}
\usepackage{pinlabel}
\usepackage{graphicx}
\usepackage{amsmath}
\usepackage{amssymb}
\usepackage{amsthm}
\usepackage{color,colortbl}
\usepackage{hyperref}
\usepackage[english]{babel}
\usepackage{epstopdf}   
\usepackage{cancel}
\usepackage[dvipsnames]{xcolor}

\usepackage{array,multirow, multicol, diagbox}

\definecolor{lightcopper}{rgb}{.93, .76, .58}

\input epsf.tex
\newcommand{\Rone}{\mathbb{Z}}
\newcommand{\Rmor}[1]{\mathbb{Z}^{#1}}
\newcommand{\Tone}[1]{\mathbb{Z}_{#1}}

\newtheorem{theorem}{Theorem}
\newtheorem{proposition}[theorem]{Proposition}
\newtheorem{corollary}[theorem]{Corollary}

\newtheorem{example}{Example}
\newtheorem{remark}[theorem]{Remark}

\renewenvironment{proof}[1][Proof]{\textit{#1.} }
{\hfill \rule{0.5em}{0.5em}}

\newcommand{\tn}{\textnormal}

\begin{document}

\title{A pattern for torsion in Khovanov homology}

\author{R. D\'{\i}az and P. M. G. Manch\'on
\footnote{R. Díaz is partially supported by Spanish Research Project PID2020-114750GB-C32/AEI/10.13039/501100011033. P. M. G. Manchón is partially supported by Spanish Research Project PID2020-117971GB-C21.}}
\maketitle


\begin{abstract}
We prove that certain specific sum of enhanced states produce torsion of order two in the Khovanov homology.
\end{abstract}


\textbf{MSC Class:} \emph{57M25, 57M27.}

\section{Introduction} \label{SectionIntroduction}
Understanding how torsion appears and what geometric meaning it has in the Khovanov homology of links and knots is a relevant problem. In 2004 Shumakovitch \cite{conjectureShumakovitch} conjectured that all links (except the trivial knot, the Hopf link and their disjoint unions and connected sums) have torsion. In \cite{Asaeda} Asaeda and Przytycki proved that certain semi-adequate links have torsion of order two if, roughly speaking, the $A$-smoothing of the semi-adequate diagram has a cycle of order odd or even, the torsion appearing then in the penultimate or antepenultimate quantum degree respectively.

Many other papers have dealt with the problem of torsion in the Khovanov homology since then. In \cite{Mukherjee}, Mukherjee showed examples that disproved some conjectures by Przytycki and Sazdanovi\'{c} relating the torsion with the braid index of the links. However, one of these conjectures is still open, and believed to be true: the Khovanov homology of a link obtained as the closure of a braid with three strands can have only torsion of order two. In \cite{ShumakovitchThinLinksHaveTorsionOnlyOfOrderTwo}
Shumakovitch proved that thin links (those whose Khovanov homology is supported in two consecutive diagonals $2i-j=s\pm 1$) have only torsion of order two. In \cite{Chandler}, Chandler, Lowrance, Sazdanovi\'{c} and Summers proved a local version of this result, and using the classification of conjugacy classes of braids with three strands by Murasugi, they provide more evidence for supporting the above conjecture. The most important tools used in that paper are the exact sequences in Khovanov homology and the spectral sequences (of Lee, Turner, Bockstein -see \cite{Chandler} for more details and references). 

In this paper we find patterns for the $A$-smoothing of a diagram that allow us to ensure the existence of torsion (see Theorem~\ref{BasicTheorem} and Corollary~\ref{MainCorollary}). By contrast with the mentioned techniques, we show specific elements that define torsion elements of order two in Khovanov homology. These patterns have certain symmetry, and they already appear in the trefoil knot or in the Borromean link. Furthermore, we show (Example~\ref{ExampleNonAlternating}) that these patterns appear in infinitely many non-alternating knots. Moreover, by contrast to what happens in \cite{Asaeda}, our elements of torsion can be found in any homological degree, and they are not restricted to semi-adequate diagrams. 

Although techniques of exact sequences could be also used to obtain our results (at least partially), we think that our approach has the benefit of constructing explicitly the torsion elements. In \cite{Kindred}, the problem of finding specific chains that define non-zero elements in the Khovanov homology is addressed. The elements found in \cite{Kindred} are some sort of traces defined as an alternating sum of enhancements of a unique Kauffman state. The torsion elements found in this paper are linear combinations of different enhanced states.

The paper is organized as follows: in Section~\ref{SectionKhovanov} we briefly review the combinatorial definition of Khovanov homology due to Viro, which we use later. In Section~\ref{SectionPattern} we show the basic patterns and prove the main theorem of the paper, Theorem~\ref{BasicTheorem}. Some examples are shown. Finally, Section~\ref{SectionBipartite} proves a practical consequence of the main theorem, Corollary~\ref{MainCorollary}. This corollary will allow to produce much more examples in which we can explicitly provide torsion elements of order two. The paper ends with some comments related to possible generalizations and to the plumbing construction.

\section{Khovanov homology} \label{SectionKhovanov}
The Khovanov homology of knots and links was introduced by Mikhail Khovanov at the end of last century (\cite{Khovanov}, \cite{Bar}). In~\cite{Viro} Viro interpreted it in terms of enhanced states of diagrams. We will use the Viro's point of view, with some simplifications of the homological and quantum/polynomial indexes taken from~\cite{DasbachLowrance}. 

Let $D$ be an oriented diagram of an oriented link $L$, with $p$ positive crossings (\includegraphics[scale=0.1]{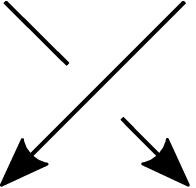}) and $n$ negative crossings (\includegraphics[scale=0.1]{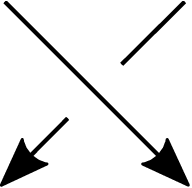}). Let $w(D)=p-n$ be the writhe of $D$.  Let $\text{Cro}(D)$ be the set of crossings of $D$. A (Kauffman) state $s$ of $D$ is a map $s:\text{Cro}(D)\rightarrow \{A,B\}$. If we smooth each crossing of $D$ according to the label of $s$ in the crossing, as shown in Figure~\ref{FigureCrossingSmoothing}, we obtain a set $sD$ of $|sD|$ disjoint simple curves, called the circles of $sD$ (see Figure~\ref{FigureEjemploEstadoRealzado}). We will draw a small chord (blue if the label is $A$, red if $B$) to remember which was the state. An enhanced state is a pair $(s,e)$ where $s$ is a state and $e$ is an assignation of signs, $-$ or $+$, to each circle of $sD$. For short, we usually write just $s$ instead of $(s,e)$ to refer to a particular enhanced state. Let $\theta(s)$ be the number of circles with sign $+$ minus the number of circles with sign $-$. We then define the homological degree $i(s)$ of an enhanced state $s$ as the number of $B$-labels of the underlying Kauffman state $s$, and its quantum (or polynomial) degree as $j(s)=i(s)+\theta(s)$. 
\begin{figure}[ht!]
	\labellist
	\pinlabel {$A$} at 140 60
	\pinlabel {$B$} at 310 60
	\endlabellist
	\begin{center}
		\includegraphics[scale=0.4]{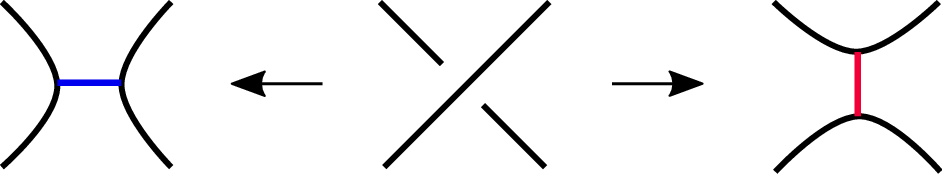}
	\end{center}
	\caption{Smoothing of a crossing according to the label}\label{FigureCrossingSmoothing}
\end{figure}

\begin{figure}[ht!]
	\labellist
	\pinlabel {$A$} at 50 135
	\pinlabel {$A$} at 193 130
	\pinlabel {$B$} at 125 -8
	\pinlabel {$-$} at 470 140
	\pinlabel {$-$} at 390 50
	\pinlabel {$sD$} at 350 150
	\endlabellist
	\begin{center}
		\includegraphics[scale=0.3]{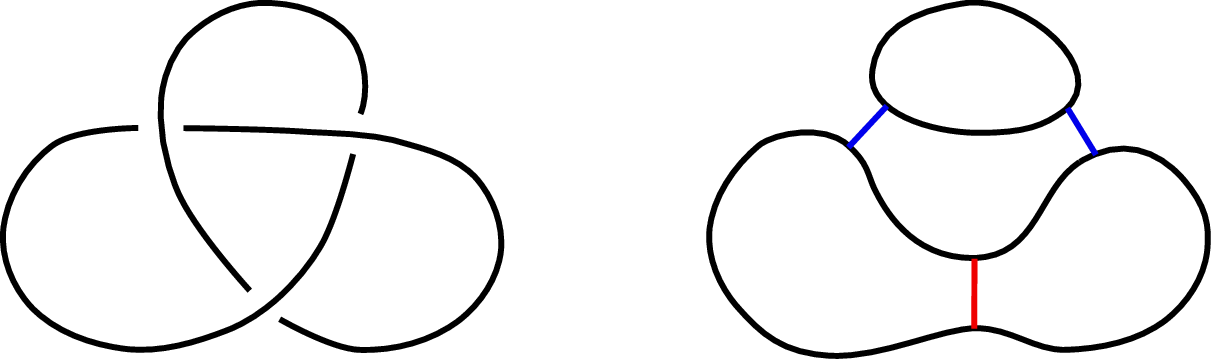}
	\end{center}
	\caption{Enhanced state $s$ with $s$ with $i(s)=1$, $\theta(s)=-2$ and $j(s)=-1$, hence $s\in C^{1,-1}(D)$. Here $|sD|=2$} \label{FigureEjemploEstadoRealzado}
\end{figure}

Let $s$ be an enhanced state with $i(s)=i$ and $j(s)=j$. An enhanced state $t$ is adjacent to $s$ if $i(t)=i(s)+1$ and $j(t)=j(s)$, the state $t$ assigns the same labels as $s$ except in one crossing $x=x(s,t)$, where $s(x)=A$ and $t(x)=B$, and $t$ assigns the same signs as $s$ to their common circles. The crossing $x(s,t)$ will be called the change crossing from $s$ to $t$. Passing then from $sD$ to $tD$ can be realized by merging two circles in one, or splitting one circle into two. Affected circles 
are those touching the crossing $x(s,t)$. The possibilities for the signs of these circles, according to the previous definition, are shown in Figure~\ref{FigurePossibilities}.
\begin{figure}[ht!]
	\begin{center}
		\includegraphics[scale=0.5]{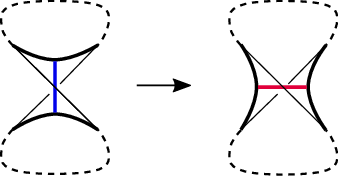} 
		\qquad \qquad
		\includegraphics[scale=0.5]{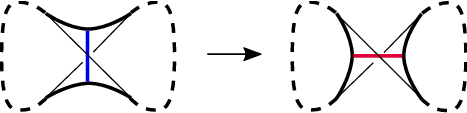}
	\end{center}
	\caption{Possible mergings: $(+,+)\mapsto +$, $(+,-)\mapsto -$ and $(-,+)\mapsto -$. Possible splittings: $+\mapsto (+,-)$ or $+\mapsto (-,+)$ and  $-\mapsto (-,-)$.}
	\label{FigurePossibilities}
\end{figure}

Let $s$ and $t$ be two enhanced states. The incident number $i(s,t)$ of $s$ over $t$ is defined as follows: if $t$ is adjacent to $s$, then $i(s,t)=(-1)^k$ where $k$ is the number of crossings of~$D$ where $s$ has a $B$-label, previous to the change crossing~$x(s,t)$; otherwise, $i(s,t)=0$. 

Let $R$ be a commutative ring with unit. Let $C^{i,j}(D)$ be the free module over~$R$ generated by the set of enhanced states~$s$ of~$D$ with $i(s)=i$ and $j(s)=j$. Numerate from $1$ to $n$ the crossings of $D$. Now fix an integer $j$ and consider the chain complex
$$
\ldots \longrightarrow C^{i,j}(D) \stackrel{d_i}{\longrightarrow} C^{i+1,j}(D) \longrightarrow \ldots 
$$
with differential $d_i(s)=\sum i(s,t)t$, where the sum runs over all the enhanced states $t$. It turns out that $d_{i+1}\circ d_i=0$. The corresponding homology modules over $R$,
$$
\underline{Kh}^{i,j}(D)=\frac{\text{ker} (d_i)}{\text{im}(d_{i-1})},
$$
are called the Khovanov homology of the diagram $D$ for degrees homological~ $i$ and polynomial~$j$. It turns out that the $R$-modules
$Kh^{h,q}(L) := \underline{Kh}^{i,j}(D)$ where $i=h+n$ and $j=q-p+2n$, are independent of the diagram $D$ and the order of its crossings; they are the Khovanov homology modules of the oriented link $L$ (\cite{Khovanov}, \cite{Bar}) as presented by Viro \cite{Viro} in terms of enhanced states, and with degrees considered as in \cite{DasbachLowrance}.

One last remark is in order. If $s$ in an enhanced state, then $d(s)=\displaystyle \sum_{x\in \text{Cro}(D)} d_x(s)$ where:
\begin{itemize}
	\item  $d_x(s)=0$ if $s(x)=B$ or $s(x)=A$ and the corresponding $A$-chord in $s_AD$ joins two different circles, both with sign $-$,
	\item $d_x(s)=(-1)^ks_{x\to B}^{+-} + (-1)^k s_{x\to B}^{-+}$ in case of splitting of a circle $+$,
	\item $d_x(s)=(-1)^ks_{x\to B}^{--}$ in case of splitting of a circle $-$,
	\item $d_x(s)=(-1)^ks_{x\to B}^{+}$ in case of merging of two circles $+$, and
	\item $d_x(s)=(-1)^ks_{x\to B}^{-}$ in case of merging two circles $+-$ or $-+$.
\end{itemize}  
Here $s_{x\to B}$ is the Kauffman state obtained from $s$ by relabeling $s_{x\to B}(x)=B$, and $k$ is the number of crossings $y\in \text{Cro}(D)$ previous to $x$ such that $s(y)=B$. The signs in the exponent of $s_{x\to B}$ refer to the signs of the circles touching the crossing $x=x(s,s_{x\to B})$.
Finally, we can consider $d_x(Z)$ for a chain $Z$, extending the definition by linearity.
\section{A pattern with monochords} \label{SectionPattern}

	Let $g$, $h$ be integers with $g,h \geq 2$. Let $D$ be an oriented diagram of an oriented link $L$. We say that $D$ is a diagram of type $D(g,h)$ if $s_AD$ has exactly $g+h$ $A$-monochords, with all their extremes in the same circle (called main circle), $g$ parallel $A$-monochords outside and $h$ parallel $A$-monochords inside, as in Figure~\ref{GeneralPattern}, and with no circles and no extra $A$-chords between each pair of consecutive parallel such $A$-monochords. 
	Note that, in principle, $s_AD$ can have other circles and bichords (chords with extremes in two different circles). We say that $D$ is a mono-circular diagram of type $D(g,h)$ if in addition $|s_AD|=1$, that is, $s_AD$ has no extra circles. Note that, in this case, the diagram is just the standard diagram of the pretzel knot $P(-1, \stackrel{(g)}{\dots},-1, h)$. 
	
	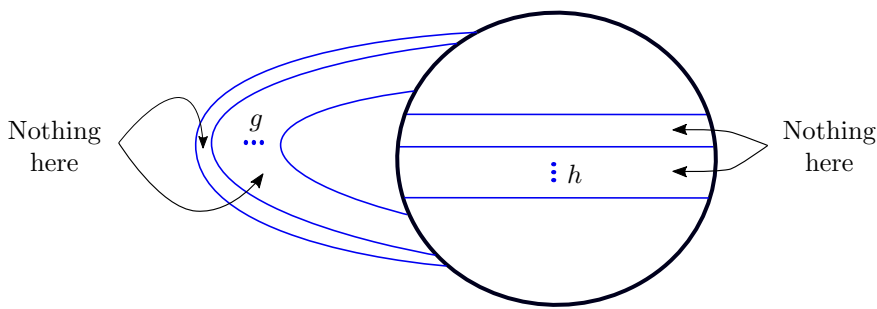
\begin{figure}[ht!]
		\labellist
		\pinlabel {$g$} at 65 87
		\pinlabel {$h$} at 215 63
		\pinlabel {\begin{tabular}{c} Nothing\\ here\end{tabular}} at -30 75
		\pinlabel {\begin{tabular}{c} Nothing\\ here\end{tabular}} at 335 75
		\endlabellist
		\begin{center}
			\includegraphics[scale=0.8]{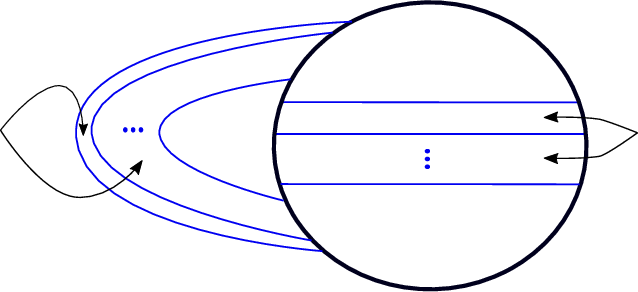}
		\end{center}
		\caption{The pattern with $g+h$ $A$-monochords. We refer to the drawn circle as the main circle.}
		\label{GeneralPattern}
	\end{figure}
	
	In this section we will work with a mono-circular diagram $D$ of type $D(g,h)$. In order to establish our first result we need to fix some notation. Let $1\leq s\leq g$ and $1\leq r\leq h$. We will write $s=s_{i_1\dots i_s;j_1\dots j_r}$ to denote the state with $r+s$ labels of type $B$, associated to the crossings that correspond to the parallel $A$-monochords $i_1, \dots, i_s$ outside, and the parallel $A$-monochords $j_1, \dots, j_r$ inside. In figures this state will be recognized by painting in red the corresponding $A$-monochords of $s_AD$. If we consider only inner monochords, then $sD$ has $r+1$ circles in a row, named circles $0$, $1$ up to $r$. We will denote by $s_{\,;j_1\dots j_r}^{l_1\dots l_t}$ ($0\leq l_1 < \dots < l_t \leq r$) the corresponding enhanced state that assign $-$ to the circles $l_1, \dots, l_t$ of $sD$. If in addition we paint red one outer monochord, the two extreme circles are merged in one; we refer to this circle as circle $0$ (note that the circle $0$ can be two different things, depending on the red chords chosen). We denote by $s^0_{i;j_1\dots j_r}$ the enhanced state that assign $-$ only to this merged circle $0$. Relevant examples are shown in Figure~\ref{XVVprima}. Basically, subscripts indicate the red chords, previous to `;' monochords outside, after `;' monochords inside, and superscripts indicate the circles with sign $-$, although, caution, we will need to break this rule in Section~\ref{SectionBipartite}, when adding extra circles and bichords to $s_AD$. 
	
	\begin{proposition} \label{proposition:dX=2V}
	Let $D$ be a mono-circular diagram of type $D(g,h)$, with $g, h~\geq~2$. Let $r$ be an integer with $1\leq r \leq h$. Let us consider the following chains: 
	$$
	X= \sum_{1\leq j_1<\dots < j_r\leq h}
	\left( 
	s_{\,;j_1\dots j_r}^{0} + s_{\,;j_1\dots j_r}^{r}
	\right)  \quad  \in C^{r,2r-1}(D), 
	$$
	$$
	V =
	\sum_{1\leq j_1<\dots<j_r\leq h} \sum_{i=1}^g 
	s_{i; j_1\dots j_r}^{0} \quad \in C^{r+1,2r-1}(D),
	$$
	and, if $r<h$,
	$$
	V'=
	\sum_{ 1\leq j_1<\dots < j_r<j_{r+1}\leq h} 
	s_{\,;j_1\dots j_r j_{r+1}}^{0\,(r+1)} \quad \in C^{r+1,2r-1}(D).
	$$
	Then:
	\begin{itemize}
		\item[(a)] $d(X)=2V$ if $r$ is odd or if $r=h$, and
		\item[(b)] $d(X)=2V+2V'$ if $r$ is even and $r<h$.
	\end{itemize}
\end{proposition}

\begin{figure}[ht!]
	\labellist
	\endlabellist
	\begin{center}
		\includegraphics[scale=0.50]{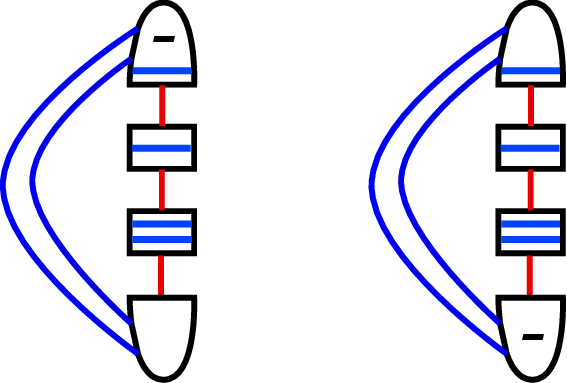}
		\qquad\qquad
		\includegraphics[scale=0.50]{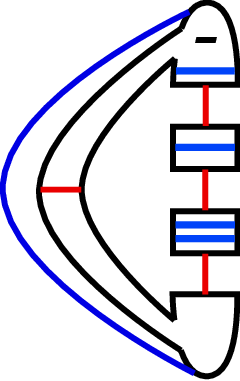}
		\qquad\qquad
		\includegraphics[scale=0.50]{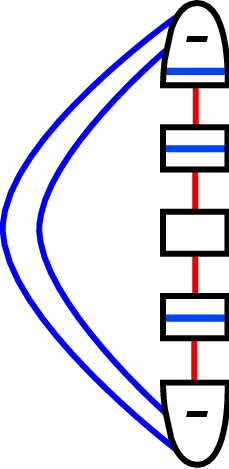}
	\end{center}
	\caption{Enhanced states $s_{\,;247}^0$ y $s_{\,;247}^r$ of $X$, $s_{2;247}^0$ of $V$, and $s_{\,;2457}^{0\, (r+1)}$ of $V'$, for $g=2$, $h=7$ and $r=3$}
	\label{XVVprima}
\end{figure}

\begin{proof}
	Let $\text{Cro}(D)=\{ x_1, \dots, x_g, y_1, \dots, y_h\}$ be the set of crossings of $D$, where the crossing $x_i$ (resp. $y_j$) corresponds to the $i$th outer (resp. $j$th inner) $A$-monochord. We will short $d_i=d_{x_i}$ (resp. $d_j=d_{y_j}$) by using the index $i\in \{ 1, \dots, g\}$ (resp. $j\in \{ 1, \dots, h\}$), hence $d_i(X)$ and $d_j(X)$ are the subchains of $d(X)$ obtained by respectively changing the outer $i$th chord and the inner $j$th chord from label $A$ to label $B$, and applying the differentiation rules. A fundamental observation is that
	$$
	d_i(s_{\,;j_1\dots j_r}^{0}) = d_i(s_{\,;j_1\dots j_r}^{r}) = s_{i; j_1\dots j_r}^{0},
	$$
	which should be clear since the outside $A$-chord $i$ in $s_{; j_1\dots j_r}D$ joins the circles $0$ and $r$, hence there is a merging $(+,-) \rightarrow -$ or $(-,+) \rightarrow -$ when this chord becomes red. 
	Using this, we have
	\begin{eqnarray*}
		d(X)&= &\sum_{ 1\leq j_1<\dots < j_r\leq h} 
		d(s_{\,;j_1\dots j_r}^{0}  +s_{\,;j_1\dots j_r}^{ r})
		\\
		&=& \sum_{1\leq j_1<\dots < j_r\leq h}
		\left(
		\sum_{i=1}^g d_i(s_{\,;j_1\dots j_r}^{0} +s_{\,;j_1\dots j_r}^{r}) + 
		\sum_{j\in\{1,\dots , h\}\setminus\{j_1,\dots, j_r\}} d_j(s_{\,; j_1\dots j_r}^{0} + s_{\,; j_1\dots j_r}^{ r})
		\right) \\
		&=& 2V + 
		\sum_{\tn{\begin{tabular}{c} $1\leq j_1<\dots < j_r\leq h$ \\ $j\in\{1,\dots , h\}\setminus\{ j_1,\dots, j_r\} $\end{tabular}}} %
		d_j(s_{\,;j_1\dots j_r}^{0}  + s_{\,;j_1\dots j_r}^{r}) \\
		&=& 2V+W,
	\end{eqnarray*}
	where we have used the previous observation to obtain $2V$ and where $W$ denotes the remaining chain. Notice that $W$ is a linear combination of states of the form $s_{\,;l_1\dots l_{r+1}}^{\eta_1\eta_2}$  with $1\leq l_1<\dots <l_{r+1}\leq h$ and $\eta_1,\eta_2\in\{0,1,\dots, {r+1}\}$. Let us see which is the coefficient of each $s_{\,;l_1\dots l_{r+1}}^{\eta_1\eta_2}$ in $W$. 
	
	Notice that 
	$$ d_j(s_{; j_1\dots j_r}^{0}) =\left\{
	\begin{array}{lcr}  
		0 & \rm if & j \in \{ j_1, \dots, j_r\}, \vspace{0.1cm}\\
		s_{\,; jj_1\dots j_r}^{01} &\rm if & 1\leq j<j_1, \vspace{0.1cm} \\
		(-1)^k ( s_{\,;  j_1\dots j_kjj_{k+1} j_r}^{0 k}+s_{\,;  j_1\dots j_kjj_{k+1} j_r}^{0\, (k+1)})
		&\rm if & j_k<j<j_{k+1}, \vspace{0.1cm} \\
		(-1)^r  ( s_{\,;  j_1\dots  j_rj}^{0 r}+s_{\,;j_1\dots j_rj}^{0\ (r+1)}) &\rm if & j_r< j. 
	\end{array}
	\right.
	$$
	Similarly,  
	$$ d_j(s_{; j_1\dots j_r}^{ {r}}) =
	\left\{
	\begin{array}{lcr}
		0 & \rm if & j \in \{ j_1, \dots, j_r\}, \vspace{0.1cm}\\  
		s_{\,; jj_1\dots j_r}^{0\,(r+1)}+s_{\,; jj_1\dots j_r}^{1\,(r+1)}  &\rm if & 1\leq j<j_1, \vspace{0.1cm} \\
		(-1)^k (s_{\,;j_1\dots j_kjj_{k+1} j_r}^{k\,(r+1)} 
		+ s_{\,;j_1\dots j_kjj_{k+1} j_r}^{(k+1)\,(r+1)})
		&\rm if & j_k<j<j_{k+1}, \vspace{0.1cm} \\
		(-1)^r   s_{\,;j_1\dots j_rj}^{r\,(r+1)}
		&\rm if & j_r<j. 
	\end{array}
	\right.
	$$
	
	In particular, the coefficient of $s^{\eta_1\eta_2}_{\,;l_1\dots l_{r+1}}$ in $W$ is zero if $\{ \eta_1, \eta_2\} \cap \{0, r+1\} = \emptyset$. 
	
	Consider now $l_1,\dots, l_{r+1}$ with $1\leq l_1<\dots <l_{r+1}\leq h$ and let $k\in\{ 1,\dots, r\}$. The only non-trivial incident numbers involving $s_{;l_1\dots l_{r+1}}^{0 k}$ are 
	$$
	i(s_{\,;l_1\dots \hat l_{k} \dots l_{r+1}}^{0}, s_{\,;l_1\dots l_{r+1}}^{0 k})=(-1)^{k-1} 
	\quad {\rm and} \quad 
	i(s_{\,;l_1\dots \hat l_{k+1} \dots l_{r+1}}^{0}, s_{\,;l_1\dots l_{r+1}}^{0 k}) 
	= (-1)^k.
	$$
	This means that the coefficient of $s_{\,;l_1\dots l_{r+1}}^{0 k}$ is zero, for any $l_1,\dots, l_{r+1}$ and $k\in\{ 1,\dots,  r\}$. 
	
	Similarly, the only non-trivial incident numbers involving $s_{\,;l_1\dots l_{r+1}}^{k\,(r+1)}$ are 
	$$
	i(s_{\,;l_1\dots \hat l_{k}\dots l_{r+1}}^{r}, s_{\,;l_1\dots l_{r+1}}^{k\, (r+1)})=(-1)^{k-1} 
	\quad {\rm and} \quad 
	i(s_{\,;l_1\dots \hat l_{k+1} \dots l_{r+1}}^{r}, s_{\,;l_1\dots l_{r+1}}^{k\,(r+1)})
	= (-1)^k.
	$$
	Hence the coefficient of $s_{\,;l_1\dots l_{r+1}}^{k\,(r+1)}$ is also zero for any $l_1,\dots, l_{r+1}$ and $k\in\{ 1,\dots,  r\}$. 
	
	It remains to analyze the coefficient of $s_{\,;l_1\dots l_{r+1}}^{0\, (r+1)}$ for any $l_1,\dots, l_{r+1}$. The only non-trivial incidence numbers involving this state are 
	$$
	i(s_{\,;l_2\dots l_{r+1}}^{r}, s_{\,;l_1l_2\dots l_{r+1}}^{0\,(r+1)}) = 1 
	\quad {\rm and}\quad 
	i(s_{\,;l_1\dots l_{r}}^{0}, s_{\,;l_1\dots l_rl_{r+1}}^{0\,(r+1)})
	= (-1)^{r}.
	$$
	Hence, if $r$ is odd, the coefficient of $s_{\,;l_1\dots l_{r+1}}^{0\,( r+1)}$ is also zero for all $l_1,\dots, l_{r+1}$, and we get finally that $W=0$, proving (a). If $r$ is even, then the coefficient of $s_{\,;l_1\dots l_{r+1}}^{0\,(r+1)}$ is equal to 2 and
	$$
	W
	= 2 \sum_{1\leq l_1<\dots <l_{r+1}\leq h} s_{\,;l_1\dots l_{r+1}}^{0\,(r+1)} 
	=2V',
	$$
	thus proving (b).
\end{proof}

	Let ${\cal B}$ be an $R$-submodule of $C^{i,j}(D)$ generated by a set of enhanced states $\{Y_k\}$ with homological degree $i$ and quantum degree $j$. Consider the projection map $\pi_{\cal B}:C(D)\rightarrow C^{i,j}(D)$. This map is the unique $R$-linear map such that, for any enhanced state $Y$, $\pi_{\cal B}(Y)=Y$ if $Y\in\{Y_k\}$ and is zero otherwise. The {\it augmentation} map $\varepsilon:C(D)\to R$ is the $R$-linear map that sends each enhanced state to 1. As in~\cite{Kindred}, maps of the form $\varepsilon \circ \pi_{\cal B}\circ d:C(D)\rightarrow R$ will be useful to prove that some enhanced states are not exact. 

\begin{theorem}\label{BasicTheorem}
	Consider a mono-circular diagram $D$ of type $D=D(g,h)$, with $g,h\geq 2$. Let $r$ be an odd integer such that $1\leq r<h$. Then the Khovanov homology group $\underline{Kh}^{r+1,2r-1}(D)$ has a torsion element of order two. 
\end{theorem}

\begin{proof} 
	We will show that the class $[V]$ of the element $V\in C^{r+1,2r-1}(D)$ defined in Proposition~\ref{proposition:dX=2V} is a torsion element of order $2$ in $\underline{Kh}^{r+1,2r-1}(D)$. 
	
	According to item (a) of Proposition~\ref{proposition:dX=2V}, there is an element $X$ with $d(X)=2V$. In particular $0 = d^2(X) = d(2V) = 2d(V)$ and, since $C^{r+2,2r-1}(D)$ is a free abelian group, it must be $d(V)=0$. Hence $V$ is a cycle and $2V$ is exact. Then, in order to see that $V$ defines a torsion element of order $2$, it remains to see that~$V$ is not an exact element. 
	
	For it, assume first that $r>1$ and consider the $R$-submodule ${\cal B}$ of $C^{r+1,2r-1}(D)$ generated by the following set $B$ of enhanced states (see Figure~\ref{GeneratorsB}): 
	$$
	B = 
	\left\{ s_{;1 \dots r+1}^{0k}\,|\,  k=1,\dots, r+1  \right\}
	\cup 
	\left\{s_{1; 2\dots r+1}^{k}  \,|\,  k=0,1,\dots , r-1 \right\}. 
	$$
	
	\begin{figure}[ht!]
		\labellist
		\endlabellist
		\begin{center}
			\includegraphics[scale=0.4]{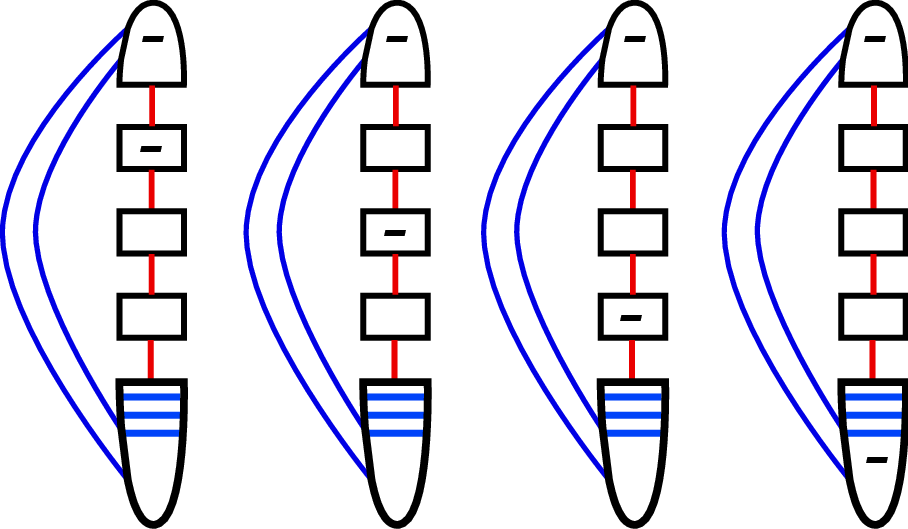}
			\qquad\quad
			\includegraphics[scale=0.4]{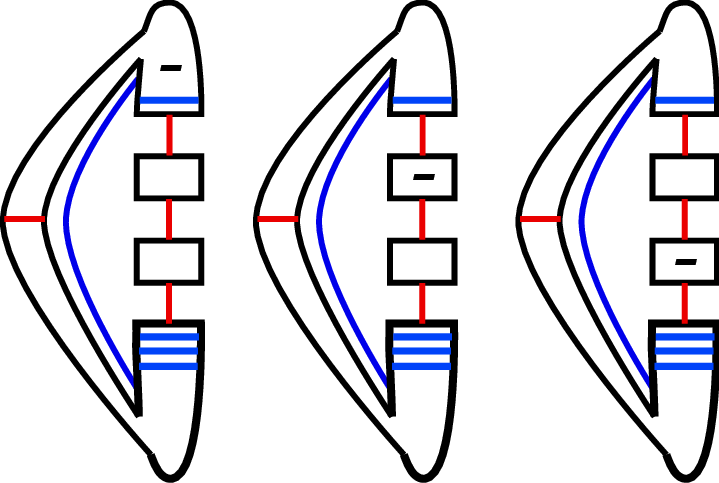}
		\end{center}
		\caption{The enhanced states, generators of ${\cal B}$, for $g=2$, $h=7$ and $r=3$}
		\label{GeneratorsB}
	\end{figure}
	
	By contradiction, suppose that $V$ is exact, that is, there is an element $Y\in C^{r,2r-1}(D)$ such that $d(Y)=V$. Notice that $\pi_{\mathcal B}(V)=s_{1; 2\dots r+1}^{ 0}$, thus $\varepsilon\circ \pi_{\mathcal B}(V)=~1$. On the other hand, we will check that the image of $\varepsilon \circ \pi_{\mathcal B}\circ d$ is contained in $2R$, which would be a contradiction. It is enough to check this for the generators of $C^{r,2r-1}(D)$, and it is also enough to do the computations modulo~$2$. This condition means that, in the matrix of $d$ (with respect to the bases formed by enhanced states), the intersection of each column with the files corresponding to the elements of~$B$ has an even number of non-zero entries. So, in the following computations, we will ignore the sign of the incidence numbers, taking only care if it is $0$ or $1$. 
	Then, if $s\in C^{r,2r-1}(D)$ is an enhanced state, we have that
	$$
	\varepsilon \circ \pi_{\cal B} \circ d (s)
	= \sum_{k=1}^{r+1}i(s,s_{;1 \dots r+1}^{ 0 k} )  
	+ \sum_{k=0}^{r-1}  i(s,s_{1; 2\dots r+1}^{k}). 
	$$
	Notice that $i(s, s_{;1 \dots r+1}^{ 0 k}) \not= 0$ implies that 
	$s= s_{;1 \dots \hat{j}\dots r+1}^{\eta}$ with $\eta \in \{0, \dots, r\}$ and $j\in \{1, \dots, r+1\}$. On the other hand, 
	$i(s, s_{1; 2\dots r+1}^{k}) \not= 0$ implies that 
	$s=s^\eta_{;2\dots r+1}$ with $\eta\in\{0, \dots, r\}$
	or 
	$s=s_{1;2\dots\hat{j}\dots r+1}$ with $j\in \{ 2, \dots, r+1\}$. 
	
	\begin{figure}[ht!]
		\labellist
		\endlabellist
		\begin{center}
		\includegraphics[scale=0.6]{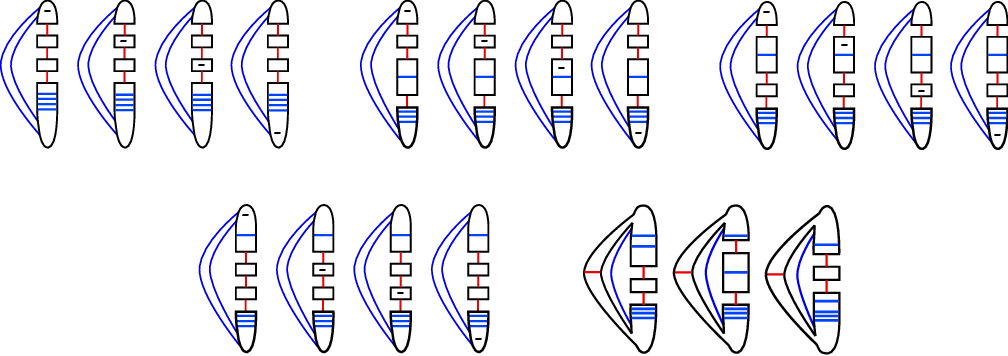}
		\end{center}
		\caption{Possible primitives of the generators in $B$, for $g=2$, $h=7$ and $r=3$. Top row: $s= s_{;1 \dots \hat{j}\dots r+1}^{\eta}$ with $\eta \in \{0, \dots, r\}$ and $j\in \{2, \dots, r+1\}$. Bottom row: $s=s^\eta_{;2\dots r+1}$ with $\eta\in\{0, \dots, r\}$ and $s=s_{1;2\dots\hat{j}\dots r+1}$ with $j\in \{ 2, \dots, r+1\}$} 
		\label{GeneratorsBPrimitivas}
	\end{figure}
	
	Thus it is enough to compute (modulo 2) the image of these states under the composition $\varepsilon \circ \pi_{\mathcal B}\circ d$, which we do next.
	
	For each $j\in \{ 2, \dots , r+1\}$,  
	$$
	\pi_{\cal B} \circ d(s_{ ;1\dots \hat j\dots  r+1}^{k})= 
	\left\{
	\begin{array}{ll}
		\pm  s_{ ;1 \dots r+1}^{0(j-1)} \pm  s_{ ;1\dots r+1}^{0j} 
		& \text{if } k=0, 
		\vspace{0.1cm} \\
		0 & {\rm otherwise.}
	\end{array}
	\right.
	$$
	
	When $j=1$ is blue (has an $A$-label),
	$$
	\pi_{\cal B} \circ d(s_{\,;2\dots r+1}^{k})= 
	\left\{
	\begin{array}{lcr}
		\pm s_{1;2\dots r+1}^{0} \pm s_{ ;12\dots r+1}^{01} 
		& {\rm if} & 
		k=0, 
		\vspace{0.1cm} \\
		\pm s_{1;2\dots r+1}^{k} \pm  s_{ ;12\dots r+1}^{0(k+1)} 
		& {\rm if} & 
		k\not=0, r, 
		\vspace{0.1cm} \\
		\pm s_{1;2\dots r+1}^{0} \pm  s_{ ;12\dots r+1}^{0\,(r+1)}  
		& {\rm if} & 
		k=r.
		\vspace{0.1cm} 
	\end{array}
	\right.
	$$
	
	Again, for $j\in \{2, \dots, r+1\}$ 
	$$
	\pi_{\cal B} \circ d(s_{1;2\dots \hat j\dots  r+1})
	= \pm s_{1 ;2\dots r+1}^{(j-1)} \pm s_{1 ;2\dots r+1}^{(j-2)}.
	$$
	It follows that $\varepsilon \circ \pi_{\cal B} \circ d (s) \in \{ -2, 0, 2\}$ for any enhanced state $s$, so the image of $\varepsilon \circ \pi_{\cal B} \circ d$ is contained in $2R$. This concludes the proof in the case $r>1$.
	
	
	\medskip
	
	Assume now that $r=1$. In this case $B$ would be the set $\{ s_{\,;12}^{01}, s_{\,;12}^{02} \} \cup \{ s_{1;2}^{0}\}$ and the above argument does not work since
	$$
	\varepsilon \circ \pi_{\cal B} \circ d(s_{1;\, }^0) 
	= \varepsilon (-s_{1;2}^0) = -1.
	$$
	The solution is, for $r=1$, to consider the new $R$-submodule ${\cal B}_1$ generated by the set of enhanced states
	$$
	B_1 = \{ s_{\,;12}^{01}, s_{\,;12}^{02} \} \cup \{ s_{1;2}^{0}\}
	\cup \{ s_{12;\,}^{02}, s_{12;\,}^{12}\}.	
	$$
	As before, if $d(s)=\sum_{t} i(s,t) t$, then $\pi_{{\cal B}_1}(d(s)) = \sum_t i(s,t) \pi_{{\cal B}_1}(t)$, hence we need to calculate $i(s,t)$ when $\pi_{{\cal B}_1}(t) \not= 0$, that is, for $t \in B_1$. We have:
	\begin{itemize}
		\item If $i(s,s_{\,;12}^{01}) \not= 0$ then $s=s_{\,;1}^0$ or $s=s_{\,;2}^0$.
		\item If $i(s,s_{\,;12}^{02}) \not= 0$ then $s=s_{\,;1}^0$ or $s=s_{\,;2}^1$.
		\item If $i(s,s_{1;2}^{0}) \not= 0$ then $s=s_{\,;2}^0$ or $s=s_{\,;2}^1$ or $s=s_{1;\,}^0$ or $s=s_{1;\,}^1$.
		\item If $i(s,s_{12;\,}^{02}) \not= 0$ then $s=s_{1;\,}^0$ or $s=s_{2;\,}^1$.
		\item If $i(s,s_{12;\,}^{12}) \not= 0$ then $s=s_{1;\,}^1$ or $s=s_{2;\,}^1$.
	\end{itemize}
	For these primitive elements, we have that (although not necessary, we easily include the right signs of the coefficients) 
	$$
	\pi_{\cal B}\left(d(s_{\,;1}^{0})\right) 
	= -s_{\,;12}^{01}-s_{\,;12}^{02},  
	\quad
	\pi_{\cal B}\left(d(s_{\,;2}^{0})\right) 
	= s_{\,;12}^{01}+s_{1;2}^{0},  
	\quad
	\pi_{\cal B}\left(d(s_{\,;2}^{1})\right) 
	= s_{\,;12}^{02}+s_{1;2}^{0},  
	$$
	and
	$$
	\pi_{\cal B} \circ d(s_{1;\,}^{0})=\pm s_{1;2}^{0}\pm s_{12;\,}^{02}, 
	\quad 
	\pi_{\cal B} \circ d(s_{1;\,}^{1})=\pm s_{1;2}^{0}\pm s_{12;\,}^{12},
	\quad
	\pi_{\cal B} \circ d(s_{2;\,}^{1}) = \pm s_{12;\,}^{02}\pm s_{12;\,}^{12}.  
	$$
	Hence, again, the image of $\varepsilon \circ \pi_{{\cal B}_1} \circ d$ is contained in $2R$, and the proof is completed. 
	\end{proof}	

%

\medskip

\begin{remark}
	If $r=g=1$ in the above theorem, then $V$ is exact. Indeed, 
	$$
	d(-s_{1; }^0) = \sum_{j=1}^h s_{1;j}^0 = V.
	$$
\end{remark}
	
\begin{example}
	The torsion in the Khovanov homology of the trefoil knot can be explained by our pattern, as Figure~\ref{PatternTrefoil} shows. For this case $g=h=2$ and $r=1$. The corresponding $s_AD$ is shown in Figure~\ref{FigureKey}. The torsion element is $V = s_{1;1}^{0} + s_{1;2}^{0} + s_{2;1}^{0} + s_{2;2}^{0}$ (Figure~\ref{FigureMainEnhancedStates}) and $X=s_{\,;1}^0 + s_{\,;1}^1 + s_{\,;2}^0+s_{\,;2}^1$ is a linear combination of enhancements of the states $x_1=s_{;1}$ and $x_2=s_{;2}$ shown in Figure~\ref{FigureStatesxonextwo}. Table~\ref{TableTrefoil} allows to check that $d^1(X)=V+V$.
	 	
	\begin{figure}[ht!] 
	\labellist
	\endlabellist
	\begin{center}
		\includegraphics[scale=0.4]{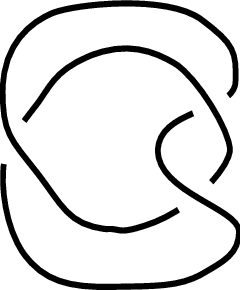}
		\qquad
		\includegraphics[scale=0.4]{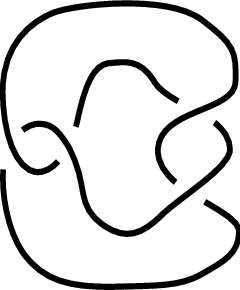}
		\qquad
		\includegraphics[scale=0.4]{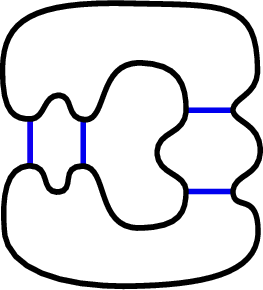}
	\end{center}
	\caption{The trefoil knot, another diagram $D$ of it and $s_AD$}
	\label{PatternTrefoil}
	\end{figure}
 	  
	\begin{figure}[ht!]
		\labellist
		\endlabellist
		\begin{center}
			\includegraphics[scale=0.3]{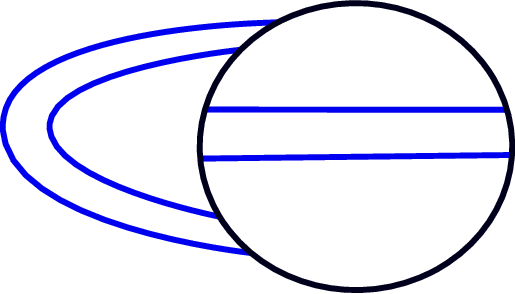}
		\end{center}
		\caption{The $A$-smoothing $s_AD$ in the case $g=h=2$}
		\label{FigureKey}
	\end{figure}
	
	\begin{figure}[ht!] 
		\labellist
		\endlabellist
		\begin{center}
			\includegraphics[scale=0.5]{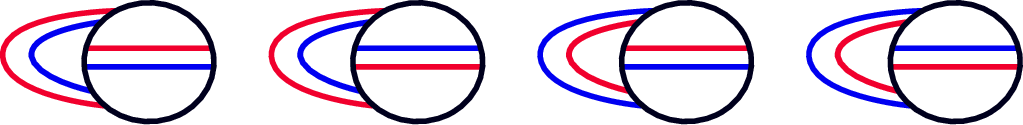}
		\end{center}
		\caption{States $s_{1;1}$, $s_{1;2}$, $s_{2;1}$ and $s_{2;2}$ for $g=h=2$ and $r=1$}
		\label{FigureMainEnhancedStates}
	\end{figure}
	
	\begin{figure}[ht!] 
		\labellist
		\endlabellist
		\begin{center}
			\includegraphics[scale=0.5]{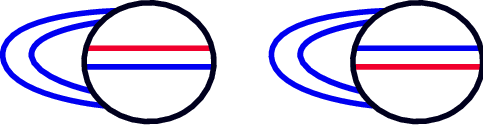}
		\end{center}
		\caption{States $x_1 = s_{\,;1}$ and $x_2 = s_{\,;2}$ for $g=h=2$ and $r=1$}
		\label{FigureStatesxonextwo}
	\end{figure}

\begin{table}
	$$
	\begin{array}{|c|c|c|c|c|c|}
		\hline
		&&\includegraphics[scale=0.7]{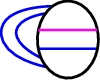}&&\includegraphics[scale=0.7]{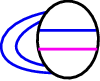}&\\
		&&\includegraphics[scale=0.7]{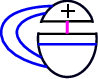}&\includegraphics[scale=0.7]{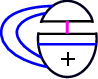}&\includegraphics[scale=0.7]{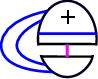}&\includegraphics[scale=0.7]{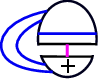}\\
		\hline
		\includegraphics[scale=0.7]{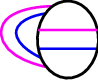}& \includegraphics[scale=0.7]{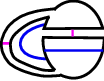}&1&1&&\\
		\hline	
		\includegraphics[scale=0.7]{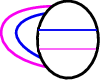}&
		\includegraphics[scale=0.7]{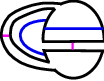}&&&1&1\\
		\hline
		\includegraphics[scale=0.7]{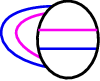}&
		\includegraphics[scale=0.7]{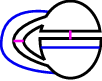}&1&1&&\\
		\hline
		\includegraphics[scale=0.7]{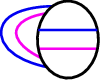}&
		\includegraphics[scale=0.7]{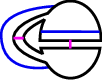}&&&1&1\\
		\hline\hline
		\includegraphics[scale=0.7]{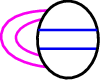}&	\includegraphics[scale=0.7]{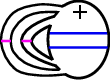}&&&&\\
		\hline
		&\includegraphics[scale=0.7]{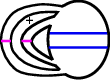}&&&&\\
		\hline
		&\includegraphics[scale=0.7]{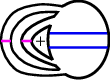}&&&&\\
		\hline \hline
		\includegraphics[scale=0.7]{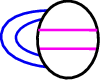}&	\includegraphics[scale=0.7]{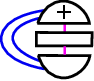}&-1&&1&\\
		\hline
		&\includegraphics[scale=0.7]{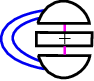}&&-1&1&\\	
		\hline
		&\includegraphics[scale=0.7]{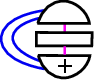}&&-1&&1\\
		\hline
	\end{array}
	$$
	\caption{This table proves that $d^1(X)=V+V$ for the trefoil diagram in Figure~\ref{PatternTrefoil}.}
	\label{TableTrefoil}
\end{table}
\end{example}

\begin{example} \label{ExampleMirrorSixOne}
	Consider the mirror image $K$ of the knot $6_1$. Figure~\ref{FigureKnotMirrorSixOne} shows this knot, a mono-circular diagram $D$ of type $D(2,5)$ for it, and the corresponding $s_AD$.
	The Khovanov homology of $K$ is shown in Table~\ref{TableKhovanovHomologyKnotMirrorSixOne}. In this table we show the homological and quantum degrees $h$ (not to be confused with the number~$h$ of inner monochords) and $q$ respectively for the knot $K$, and, in blue with large numbers, the corresponding degrees $i$ and $j$ for the diagram~$D$. Since $D$ has $p=2$ positive crossings and $n=5$ negative crossings, it follows that $Kh^{h,q}(L) = \underline{Kh}^{i,j}(D)$ with $i=h+n=h+5$ and $j=q-p+2n=q+8$.
	By Theorem~\ref{BasicTheorem}, for mono-circular diagrams of type $D(2,5)$ and $r=1$ and $r=3$, we obtain the $\mathbb{Z}_2$ torsion in the rows corresponding to the polynomial degrees $j=1$ and $j=5$ respectively. This example also shows the necessity of the hypothesis $r<h$ in Theorem~\ref{BasicTheorem}; otherwise there would be torsion in $C^{i,j}(D)=C^{r+1,2r-1}(D)=C^{6,9}(D)$, which is not true.
	\begin{figure}[ht!]
		\labellist
		\endlabellist
		\begin{center}
			\includegraphics[scale=0.35]{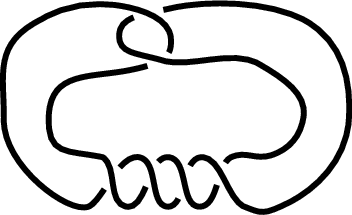}
			\qquad  
			\includegraphics[scale=0.35]{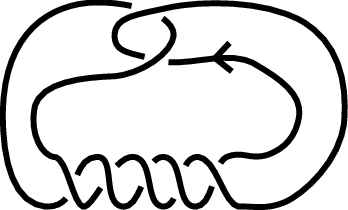}
			\qquad  
			\includegraphics[scale=0.35]{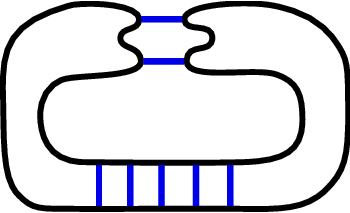}
		\end{center}
		\caption{Mirror image $K=\overline{6_1}$ of the knot $6_1$, a mono-circular diagram $D$ of~$K$ of type $D(2,5)$, and the corresponding $s_AD$}
		\label{FigureKnotMirrorSixOne}
	\end{figure}
	
	\begin{table}
		\begin{center}
		\setlength\extrarowheight{2pt}
		\begin{tabular}{|c||c|c|c|c|c|c|c|}
			\hline
			\backslashbox{\!{\color{blue}j} \tiny{$q$}\!}{\!{\color{blue}i} \tiny{$h$}\!} & {\color{blue}1} \tiny{$-4$} & {\color{blue}2} \tiny{$-3$} & {\color{blue}3} \tiny{$-2$} & {\color{blue}4} \tiny{$-1$} & {\color{blue}5} \tiny{$0$} & {\color{blue}6} \tiny{$1$} & {\color{blue}7} \tiny{$2$} \\
			\hline
			\hline
			{\color{blue}13} \tiny{$-5$}&&&&&&&$\Rone$\\
			\hline
			{\color{blue}11} \tiny{$3$} \,\, &&&&&&&$\Tone{2}$\\
			\hline
			{\color{blue}9} \tiny{$1$} &&&&&$\Rmor{2}$&{\cellcolor{lightgray}{}}$\Rone$&\\
			\hline
			\, {\color{blue}7} \tiny{$-1$}&&&&$\Rone$&$\Rone\oplus\Tone{2}$&&\\
			\hline
			\, {\color{blue}5} \tiny{$-3$}&&&&$\Rone\oplus{\color{red}\Tone{2}}$ &&&\\
			\hline
			\, {\color{blue}3} \tiny{$-5$}&&$\Rone$&$\Rone$&&&&\\
			\hline
			\, {\color{blue}1} \tiny{$-7$}&&{\color{red}$\Tone{2}$}&&&&&\\
			\hline
			\,{\color{blue}-1} \tiny{$-9$}&$\Rone$&&&&&&\\
			\hline
		\end{tabular}
		\end{center}
		\caption{Khovanov homology of the knot $K=\overline{6_1}$. The large blue indexes $i$ and $j$ correspond to the Khovanov homology of the diagram $D$ of $K$}
		\label{TableKhovanovHomologyKnotMirrorSixOne}
	\end{table}
\end{example}

\section{Extending the theorem to the bipartite case} \label{SectionBipartite}
Let $D$ be a diagram of type $D(g,h)$, and suppose that the graph defined by $s_AD$ minus the monochords, is bipartite. We want to construct elements of torsion of order two, hence extending Theorem~\ref{BasicTheorem}

We start by generalizing Proposition~\ref{proposition:dX=2V}. Assume that $r$ is an odd integer and $r<h$. We first explain a new notation for dealing with the signs of the circles:
\begin{itemize}
	\item We denote by $s=s_{i_1\dots i_s;j_1\dots j_r;l_1\dots l_t}$ the state that assigns a $B$-label to the crossings that correspond to the outer $A$-monochords $i_1, \dots, i_s$, the inner $A$-monochords $j_1, \dots, j_r$ and the $A$-bichords $l_1, \dots, l_t$, all seen in $s_AD$.
	When writing $s_{\,;\dots i\hat{j}k\dots;\,}$ we are assuming that the label associated to the crossing that corresponds to the $j$th inner monochord is still $A$.
	
	\item In $s_{\,;j_1,\dots,j_r;\,}D$ there are $r+1$ circles in a column, numbered $0, 1, \dots, r$. We will refer to circles $0$ and $r$ as the {\it extreme circles}; circles $1, \dots, r-1$ will be called {\it $H$-circles}. We will write
	$$
	s_{\,;j_1,\dots,j_r;\,}^{k+} \, \in C^{r, 2r-|s_AD|}(D), \quad k \in  \{0, r\},
	$$
	for the enhancement of $s_{\,;j_1,\dots,j_r;\,}$ that assigns $+$ to the extreme circle~$k$, and in addition to all the $H$-circles $1, \dots, r-1$ (Figure~\ref{FigureNotacionX}, first two pictures).
	
	In addition, we have now  external circles $C_\alpha$, indexed say by a set $\Lambda$ with $\Lambda \cap \{0,1,\dots, r\} = \emptyset$. Note that $|\Lambda|+1=|s_AD|$. We will write
	$$
	s_{\,;j_1,\dots,j_r;\,}^{\alpha+} \, \in C^{r, 2r-|s_AD|}(D), \quad \alpha \in \Lambda,
	$$
	for the enhancement of $s_{\,;j_1,\dots,j_r;\,}$ that assigns $+$ to the external circle~$C_\alpha$, and in addition to all the $H$-circles $1, \dots, r-1$ (Figure~\ref{FigureNotacionX}, last three pictures). 
	
	In both cases, $s_{\,;j_1,\dots,j_r;\,}^{k+}$ with $k \in  \{0, r\}$ or $s_{\,;j_1,\dots,j_r;\,}^{\alpha+}$ with $\alpha \in \Lambda$, we have
	$$
	\theta = (r-1 + 1)-( 1+ |\Lambda|) = r-|s_AD|, 
	\quad 
	j = i+\theta = 2r-|s_AD|.
	$$ 
		
	\begin{figure}[ht!]
		\labellist
		\endlabellist
		\begin{center}
			\includegraphics[scale=0.7]{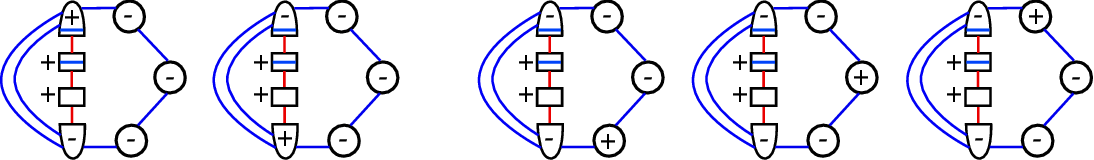}
		\end{center}
		\caption{Enhanced states $s_{\,;245;\,}^{0+}$, $s_{\,;245;\,}^{r+}$ and $s_{\,;245;\,}^{\alpha+}$ for $\alpha = 4,5,6$, when $g=2$, $h=5$ and $r=3$}
		\label{FigureNotacionX}
	\end{figure}
	
	\item We will write 
	$$
	s_{i; j_1\dots j_r;\,}^{0-} \, \in C^{r+1,2r-|s_AD|}(D)
	$$
	for the enhancement of $s_{i; j_1\dots j_r;\,}$ with sign $-$ in all external circles, and in the circle $0$ obtained by merging the extreme circles in $s_{\,;j_1,\dots,j_r;\,}$. Hence
	$$
	\theta = (r-1)-(1+|\Lambda|) = r-1-|s_AD|, \quad 
	j = i+\theta = (r+1) + \theta = 2r-|s_AD|.
	$$
	\begin{figure}[ht!]
		\labellist
		\endlabellist
		\begin{center}
			\includegraphics[scale=0.3]{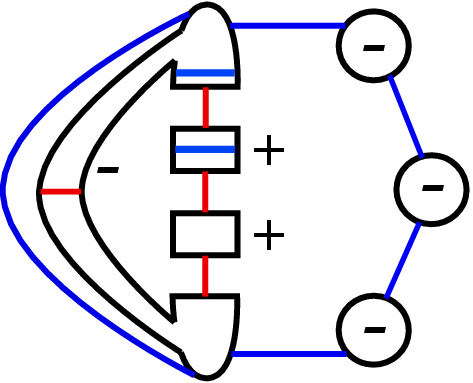}
		\end{center}
		\caption{Enhanced state $s_{2;245;\,}^{0-}$ when $g=2$, $h=5$ and $r=3$}
		\label{FigureNotacionV}
	\end{figure}
\end{itemize}

We are now ready for stating and proving the main result of this paper:

\begin{corollary} \label{MainCorollary}
	Let $D$ be a diagram of type $D(g,h)$ and assume that the graph defined by $s_AD$ without the monochords, is bipartite. Suppose that $r$ is odd and $r<h$. Let us consider the chains
	$$
	X= \sum_{1\leq j_1<\dots < j_r\leq h}
	\left( s_{\,;j_1\dots j_r;\,}^{0+} + s_{\,;j_1\dots j_r;\,}^{r+} + \sum_{\alpha \in \Lambda} (-1)^{l_\alpha} s_{\,;j_1\dots j_r;\,}^{\alpha+} \right)
	\, \in C^{r,2r-|s_AD|}(D)
	$$
	and
	$$
	V =
	\sum_{1\leq j_1<\dots < j_r\leq h} 
	\, \sum_{i=1}^g s_{i;j_1\dots j_r;\,}^{0-} 
	\, \in C^{r+1,2r-|s_AD|}(D),
	$$
	where, for each $\alpha \in \Lambda$, the integer $l_\alpha$ is the length of any path from the external circle $C_\alpha$ to the main circle in the graph defined by $s_AD$ without the monochords.
		 
	Then $d(X)=2V$ and $V$ defines a torsion element of order two in the Khovanov homology of $D$.
\end{corollary}

Notice that, when $D$ is just a mono-circular diagram of type $D(g,h)$, the chains defined in Corollary~\ref{MainCorollary} coincide with those defined in Proposition~\ref{proposition:dX=2V}. 

\begin{proof}
	
	As in the proof of Proposition~\ref{proposition:dX=2V} and Theorem~\ref{BasicTheorem}, we will show that $d(X)=2V$ and that $X$ is not exact, hence $V$ defines an element of torsion, of order two, in $\underline{Kh}^{r+1,2r-|s_AD|}(D)$. 
		
	Let $\text{Cro}(D) = \{ x_1, \dots, x_g, y_1, \dots, y_h, z_1, \dots,z_{|I|}\}$ be the set of crossings of $D$, where $x_i$ corresponds to the outer $i$th $A$-monochord, $y_j$ to the inner $j$th $A$-monochord and $z_l$ to the $l$th bichord. Recall that the set $\Lambda$ indexes the external circles. To shorten notation, from now on we will write $J$ instead of $1\leq j_1 < \dots <j_r \leq h$.
	
	Let $E=\{z_1, \dots, z_{|I|}\}$. We define subsets $\text{TOP}, \text{BOT} \subset E$ as follows: a crossing $z\in E$ belongs to $\text{TOP}$ (resp. $\text{BOT}$) if the corresponding bichord has one of its ends in the main circle, above (resp. under) the parallel inner $A$-monochords. Note that, if $z\in (\text{TOP}\, \cup \, \text{BOT})$ has its other end in a circle $C_\alpha$, then $l_\alpha =1$.  
	 	
	We start by rearranging the chain $d(X)$ conveniently:
	$$
	\begin{array}{rclc}
		d(X)
		&=& \displaystyle \sum_{J}
		\sum_{x\in \{x_1, \dots , x_g,y_1, \dots , y_h\}} 
		\left( d_x(s^{0+}_{\,;j_1\dots j_r;\,}) + d_x(s^{r+}_{\,;j_1\dots j_r;\,})\right) & (1)\\
		&& \displaystyle + \sum_{J}
		\sum_{x\in \{z_1, \dots , z_{|I|}\}} 
		\left( d_x(s^{0+}_{\,;j_1\dots j_r;\,}) + d_x(s^{r+}_{\,;j_1\dots j_r;\,})\right) & (2)\\
		&& \displaystyle + \sum_{J}
		\sum_{\alpha \in \Lambda} (-1)^{l_\alpha} 
		\sum_{x\in \{x_1, \dots , x_g\}} d_x(s^{\alpha +}_{\,;j_1\dots j_r;\,}) & (3) \\
		&& \displaystyle + \sum_{J}
		\sum_{\alpha \in \Lambda} (-1)^{l_\alpha} 
		\sum_{x\in \{y_1, \dots , y_h\}} d_x(s^{\alpha +}_{\,;j_1\dots j_r;\,}) & (4)\\
		&& \displaystyle
		+ \sum_{J}
		\sum_{\alpha \in \Lambda} (-1)^{l_\alpha} 
		\sum_{x\in E\setminus (\text{TOP}\,\cup\,\text{BOT})} d_x(s^{\alpha +}_{\,;j_1\dots j_r;\,}) & (5)\\
		&& \displaystyle
		+ \sum_{J}
		\sum_{\alpha \in \Lambda} (-1)^{l_\alpha} 
		\sum_{x\in \text{TOP}} d_x(s^{\alpha +}_{\,;j_1\dots j_r;\,}) & (6) \\
		&& \displaystyle + \sum_{J}
		\sum_{\alpha \in \Lambda} (-1)^{l_\alpha} 
		\sum_{x\in \text{BOT}} d_x(s^{\alpha +}_{\,;j_1\dots j_r;\,}). & (7)
	\end{array}
	$$
	The sum $(3)$ is equal to zero: in fact, each $d_x(s^{\alpha +}_{\,; j_1\dots j_r;\,}) = 0$ since it would be a merging of two circles $-$.  
	
	The sum $(4)$ is also equal to zero: in this case, for each $\alpha \in \Lambda$ (the sign $+$ is fixed in the circle $C_\alpha$), we have that
	$$
	\sum_{J}
	\sum_{x\in \{y_1, \dots , y_h \}} d_x(s^{\alpha +}_{\,;j_1\dots j_r;\,}) = 0.
	$$
	In order to check this, note that, if $x\in \{y_1, \dots , y_h \}$, then in $d_x(s^{\alpha +}_{\,;j_1\dots j_r;\,})$ only appear elements of the form $s^{\alpha+\, 0k(r+1)-}_{\,;t_1\dots t_{r+1};\,}$ with $k=1, \dots, r$
	(see Figure~\ref{FigureAuxiliarPrueba}).
		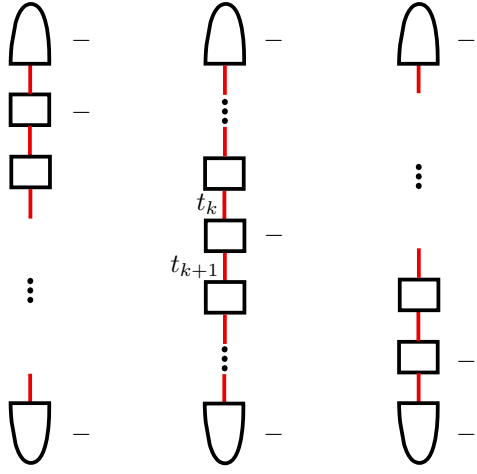
\begin{figure}[ht!]
		\labellist
		\pinlabel{$-$} at 23 135
		\pinlabel{$-$} at 84 135
		\pinlabel{$-$} at 145 135
		\pinlabel{$-$} at 23 112
		\pinlabel{$t_k$} at 63 83		
		\pinlabel{$-$} at 145 33
		\pinlabel{$t_{k+1}$} at 59 63		
		\pinlabel{$-$} at 84 73
		\pinlabel{$-$} at 23 10
		\pinlabel{$-$} at 84 10
		\pinlabel{$-$} at 145 10
		\endlabellist
		\begin{center}
			\includegraphics[scale=1.2]{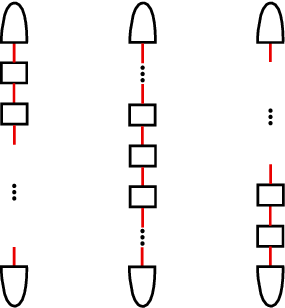}
		\end{center}
		\caption{Enhanced states $s^{\alpha+\, 0k(r+1)-}_{\,;t_1\dots t_{r+1};\,}$ for $k=1$ (left), $k\in\{1,\dots, r-1\}$ (central) and $k=r$ (right). Circles without sign are circles~$+$. Blue inner monochords have not been drawn although they are relevant. Outer monochords and external circles, not shown, are not relevant here.}
		\label{FigureAuxiliarPrueba}
	\end{figure}
	Each of these enhanced states appears twice and with opposite signs in the above sum when differentiating, via the crossings $y_j$, two enhanced states of the form $s^{\alpha +}_{\, ;j_1\dots j_r}$. Precisely,
	\begin{itemize}
	\item if $s=s^{\alpha+\, 01(r+1)-}_{\,;t_1\dots t_{r+1};\,}$, then (see Figure~\ref{FigureAuxiliarDos})
	$$
	d_{t_1}(s^{\alpha+}_{\,;t_2\dots t_{r+1};\,}) 
	= s
	\quad \text{ and } \quad
	d_{t_2}(s^{\alpha +}_{\, ; t_1t_3\dots t_{r+1};\,}) 
	= -s_{\, ; t_1t_2t_3\dots t_{r+1};\,}^{\alpha+\, 02(r+1)-} - s.
	$$
	\item If $s=s^{\alpha+ 0k(r+1)-}_{\,;t_1\dots t_{r+1};\,}$ with $k=2, \dots, r-1$, then (see Figure~\ref{FigureAuxiliarUnoDosAuxiliarUnoTres})
	$$
	d_{t_k}(s^{\alpha+}_{\,;t_1\dots \hat{t}_k \dots t_{r+1};\,})
	= (-1)^{k-1}s 
	+ (-1)^{k-1} s^{\alpha+\ 0(k-1)(r+1)-}_{\,;t_1\dots t_{r+1};\,}
	$$
	and
	$$
	d_{t_{k+1}}(s^{\alpha+}_{\,;t_1\dots \hat{t}_{k+1}\dots t_{r+1};\,})
	= (-1)^{k} s^{\alpha+\, 0(k+1)(r+1)-}_{\,;t_1\dots t_{r+1};\,} + (-1)^k s.
	$$
	\item And if $s=s^{\alpha+\, 0r(r+1)-}_{\,;t_1\dots t_{r+1};\,}$, then
	$$
	d_{t_r}(s^{\alpha +}_{\,;t_1\dots \hat{t}_r t_{r+1};\,}) 
	= s+s_{\,;t_1\dots t_rt_{r+1};\,}^{\alpha+\, 0(r-1)(r+1)-}
	\quad \text{ and } \quad 
	d_{t_{r+1}}(s^{\alpha+}_{\,;t_2\dots t_r};\,) 
	= -s.	
	$$
	\end{itemize}
	
	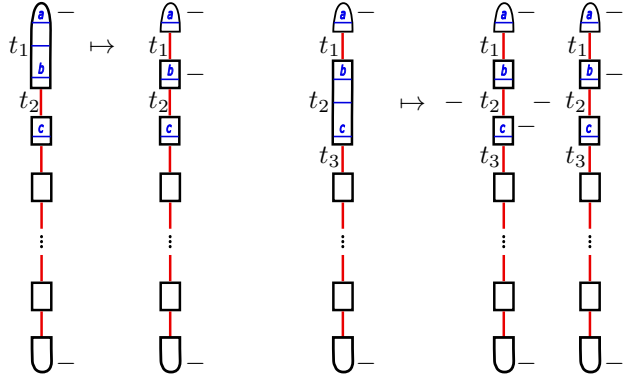
\begin{figure}[ht!]
		\labellist
		\pinlabel {$t_1$} at -10 317
		\pinlabel {$t_1$} at 125 317
		\pinlabel {$t_1$} at 290 317
		\pinlabel {$t_1$} at 445 317
		\pinlabel {$t_1$} at 530 317
		\pinlabel {$t_2$} at 0 263
		\pinlabel {$t_2$} at 125 263
		\pinlabel {$t_2$} at 280 263
		\pinlabel {$t_2$} at 445 263
		\pinlabel {$t_2$} at 528 263
		\pinlabel {$t_3$} at 290 210
		\pinlabel {$t_3$} at 445 210
		\pinlabel {$t_3$} at 528 210
		\pinlabel{$-$} at 410 263
		\pinlabel{$-$} at 495 263
		\pinlabel{$\mapsto$} at 70 315
		\pinlabel{$\mapsto$} at 370 260
		\pinlabel{$-$} at 35 350
		\pinlabel{$-$} at 35 10
		\pinlabel{$-$} at 160 350
		\pinlabel{$-$} at 160 10
		\pinlabel{$-$} at 160 290
		\pinlabel{$-$} at 325 350
		\pinlabel{$-$} at 325 10
		\pinlabel{$-$} at 480 350
		\pinlabel{$-$} at 480 10
		\pinlabel{$-$} at 480 240
		\pinlabel{$-$} at 565 350
		\pinlabel{$-$} at 565 10
		\pinlabel{$-$} at 565 290
		\endlabellist
		\begin{center}
			\includegraphics[scale=0.4]{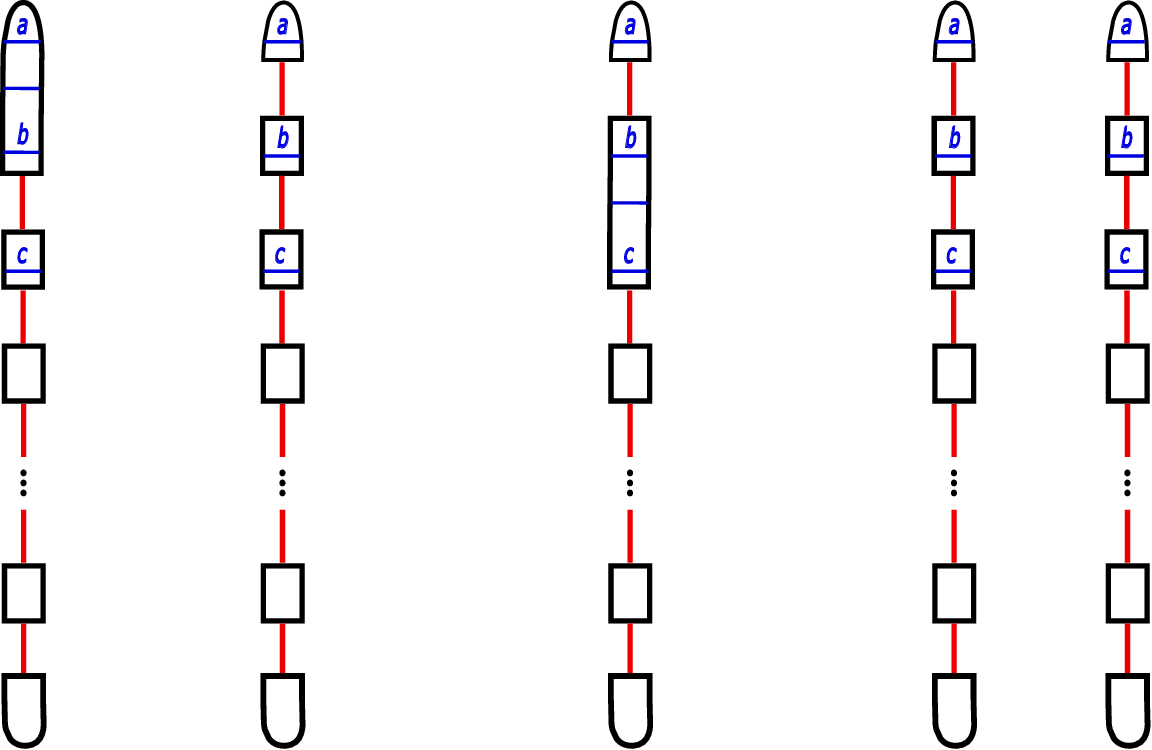}
		\end{center}
		\caption{$d_{t_1}(s^{\alpha+}_{\,;t_2\dots t_{r+1};\,})\!=\!s=s^{\alpha+\, 01(r+1)-}_{\,;t_1\dots t_{r+1};\,}$, $d_{t_2}(s^{\alpha +}_{\,;t_1t_3\dots t_{r+1};\,})\!=\!-s_{\,;t_1t_2\dots t_{r+1};\,}^{\alpha+\,02(r+1)-}\!-\!s$}
		\label{FigureAuxiliarDos}
	\end{figure}
				
	\begin{figure}[ht!]
	\labellist
	\pinlabel {$t_{k-1}$} at -20 350
	\pinlabel {$\hat{t}_k$} at -20 277
	\pinlabel {$t_{k+1}$} at -20 202
	\pinlabel {$t_{k-1}$} at 180 350
	\pinlabel {${t}_k$} at 195 277
	\pinlabel {$t_{k+1}$} at 180 202
	\pinlabel {$t_{k-1}$} at 325 350
	\pinlabel {${t}_k$} at 340 277
	\pinlabel {$t_{k+1}$} at 325 202
	\pinlabel{$\mapsto$} at 90 450 
	\pinlabel{$\epsilon$} at 150 450 
	\pinlabel{$+\,\epsilon$} at 290 450 
	\pinlabel{$-$} at 55 530
	\pinlabel{$-$} at 55 10
	\pinlabel{$-$} at 250 530
	\pinlabel{$-$} at 250 10
	\pinlabel{$-$} at 250 255
	\pinlabel{$-$} at 395 530
	\pinlabel{$-$} at 395 10
	\pinlabel{$-$} at 395 325

	\pinlabel {$t_{k-1}$} at 570 350
	\pinlabel {$t_k$} at 580 275
	\pinlabel {$\hat{t}_{k+1}$} at 550 203
	\pinlabel {$t_{k-1}$} at 755 350
	\pinlabel {$t_k$} at 770 275
	\pinlabel {$t_{k+1}$} at 765 203
	\pinlabel {$t_{k-1}$} at 920 350
	\pinlabel {$t_k$} at 940 275
	\pinlabel {$t_{k+1}$} at 920 203
	\pinlabel{$\mapsto$} at 670 450 
	\pinlabel{$-\epsilon$} at 720 450 
	\pinlabel{$-\,\epsilon$} at 870 450 
	\pinlabel{$-$} at 640 530
	\pinlabel{$-$} at 640 10
	\pinlabel{$-$} at 830 530
	\pinlabel{$-$} at 830 10
	\pinlabel{$-$} at 830 180
	\pinlabel{$-$} at 995 530
	\pinlabel{$-$} at 995 10
	\pinlabel{$-$} at 995 255
	\endlabellist
	\begin{center}
		\includegraphics[scale=0.3]{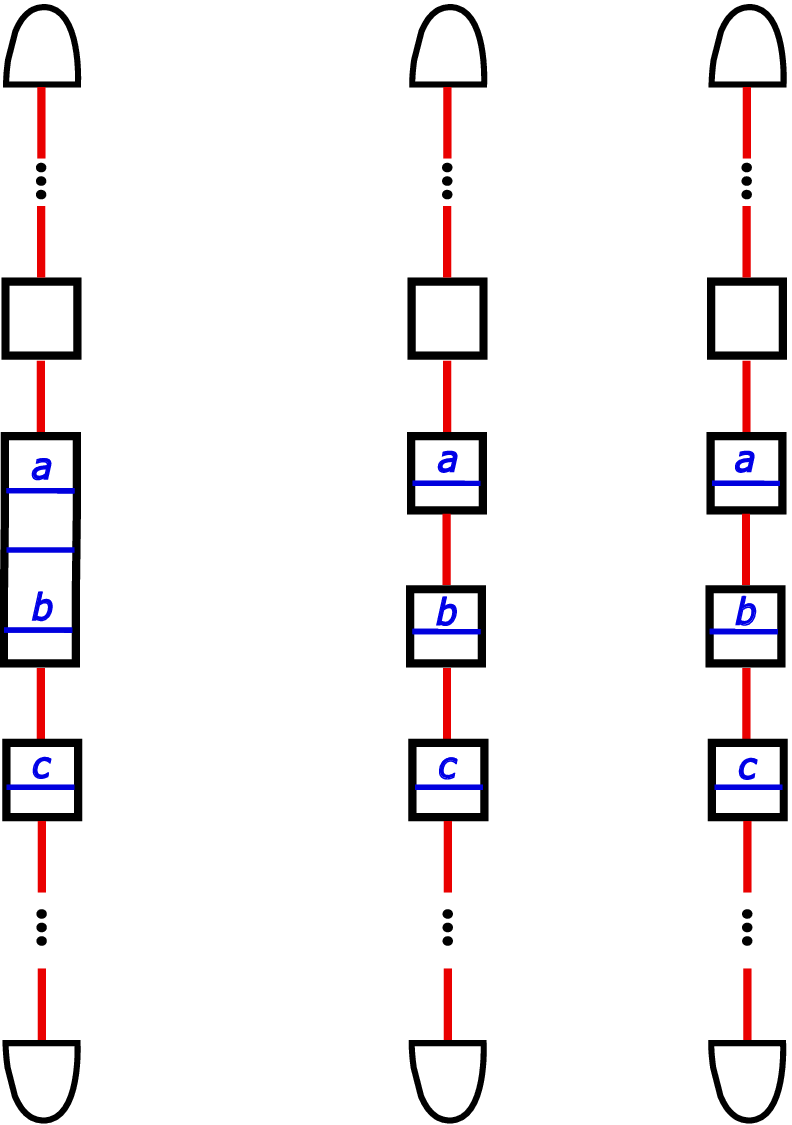}
		\qquad\qquad \qquad
		\includegraphics[scale=0.3]{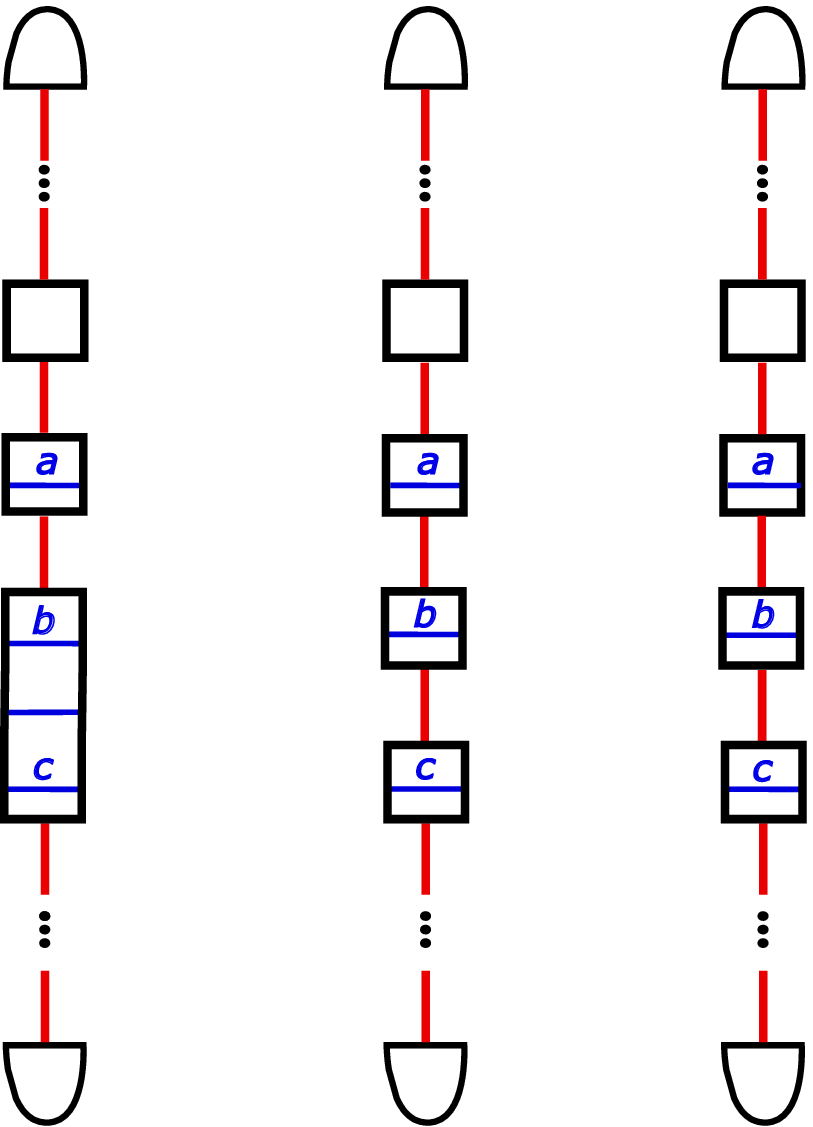}
	\end{center}
	\caption{$d_{t_k}(s^{\alpha+}_{\,;t_1\dots \hat{t}_k \dots t_{r+1};\,})
		= (-1)^{k-1}s 
		+ (-1)^{k-1} s^{\alpha+\ 0(k-1)(r+1)-}_{\,;t_1\dots t_{r+1};\,}$ and $d_{t_{k+1}}(s^{\alpha+}_{\,;t_1\dots \hat{t}_{k+1}\dots t_{r+1};\,})
		= (-1)^{k} s^{\alpha+\, 0(k+1)(r+1)-}_{\,;t_1\dots t_{r+1};\,} + (-1)^k s$ where $s=s^{\alpha+ 0k(r+1)-}_{\,;t_1\dots t_{r+1};\,}$}
	\label{FigureAuxiliarUnoDosAuxiliarUnoTres}
	\end{figure}

	The sum $(5)$ is equal to zero: now, fixed a particular sequence $J=(1\leq j_1<\dots < j_r \leq h)$, we have that
	$$
	\sum_{\alpha \in \Lambda} (-1)^{l_\alpha} 
	\sum_{x\in E\setminus (\text{TOP}\,\cup\,\text{BOT})} 
	d_x(s^{\alpha +}_{\, ;j_1\dots j_r;\,})
	= 0.
	$$
	This is true since there is a cancellation by pairs. To see this, suppose that the bichord corresponding to a crossing $x\in E\setminus (\text{TOP}\,\cup\,\text{BOT})$ joins circles~$C_\alpha$ and~$C_\beta$.  We have the following possibilities: when both circles have sign $-$, then $d_x$ gives zero. If $C_\alpha$ or $C_\beta$ is $+$, both circles merge into a single one, with sign $-$; if $C_\alpha$ is $+$, we obtain the coefficient $(-1)^{l_\alpha}$ and if $C_\beta$ is $+$ we obtain~$(-1)^{l_\beta}$. Since $(-1)^{l_\beta}= -(-1)^{l_\alpha}$, they cancel.
	
	We now pay attention to sum (2), and observe that it is equal to
	$$
	\displaystyle 
	\sum_{J}
	\sum_{x\in \text{TOP}} 
	d_x(s^{0+}_{\,;j_1\dots j_r;\,}) 
	+ 
	\sum_{J}
	\sum_{x\in \text{BOT}} 
	d_x(s^{r+}_{\, ;j_1\dots j_r;\,}),
	$$
	since if $x\in E\setminus (\text{TOP} \cup \text{BOT})$, then the corresponding bichord is joining two circles with signs $-$ in the enhanced state $s^{0+}_{\,;j_1\dots j_r;\,}$ (and also in $s^{r+}_{\,;j_1\dots j_r;\,}$), and so $d_x$ gives zero.
	
	Then, it turns out that the first summand in the expression above cancels with the sum (6). Precisely,
	for a fixed sequence $J=(1\leq j_1 < \dots < j_r \leq h)$ it turns out that 
	$$
	\sum_{x\in \text{TOP}} 
	d_x(s^{0+}_{\,;j_1\dots j_r;\,}) 
	+ 
	\sum_{\alpha \in \Lambda} (-1)^{l_\alpha} 
	\sum_{x\in \text{TOP}} d_x(s^{\alpha +}_{\, ;j_1\dots j_r;\,}) 
	= 0.
	$$ 
	Indeed, if $x\in \text{TOP}$, then $d_x(s^{\alpha +}_{\, ;j_1\dots j_r;\,})=0$ except if the circle $C_\alpha$ with sign $+$ is joined to the circle $0$ by the bichord associated to the crossing $x$, and in this case $d_x(s^{\alpha +}_{\, ;j_1\dots j_r;\,})$ produces a merged circle with sign $-$, and coefficient $(-1)^{l_\alpha} = (-1)^1=-1$, which cancels with the $d_x(s^{0+}_{\,;j_1\dots j_r;\,})$ in the first sum. 
	
	Analogously, the second summand cancels with the sum in (7), that is,
	$$
	\sum_{J}
	\sum_{x\in \text{BOT}} 
	d_x(s^{r+}_{\,;j_1\dots j_r;\,}) 
	+
	\sum_{J}
	\sum_{\alpha \in \Lambda} (-1)^{l_\alpha} 
	\sum_{x\in \text{BOT}} d_x(s^{\alpha +}_{\, ;j_1\dots j_r;}) 
	= 0.
	$$
	Hence we have discovered that only the summand (1) does not vanish, and it is
	$$
	d(X) 
	= \displaystyle \sum_{J}
	\sum_{x_1\in \{x_, \dots , x_g,y_1, \dots , y_h\}} 
	\left( d_x(s^{0+}_{\,;j_1\dots j_r}) + d_x(s^{r+}_{\, ;j_1\dots j_r})\right)
	= 2V
	$$
	by an argument similar to the proof of Proposition~\ref{proposition:dX=2V}.
	This proves that $d(X)=2V$.
	
	In order to complete the proof, we need to check that $2V$ is not an exact element. The proof is partially similar to that of Theorem~\ref{BasicTheorem}. In any case, we write the complete details. 
	
	We consider the submodule ${\cal B}$ of $C(D)$ generated by the elements (see Figure~\ref{GeneratorsBWithExternalCircles})
	$$
	B = \{ s^{0k-}_{\,;1, \dots, r+1;\,} \,/\, k=1, \dots, r+1\} 
	\cup 
	\{ s^{k-}_{1;2,\dots, r+1;\,} \,/\, k=0, \dots, r-1\}.
	$$
	\begin{figure}[ht!]
		\labellist
		\endlabellist
		\begin{center}
			\includegraphics[scale=0.25]{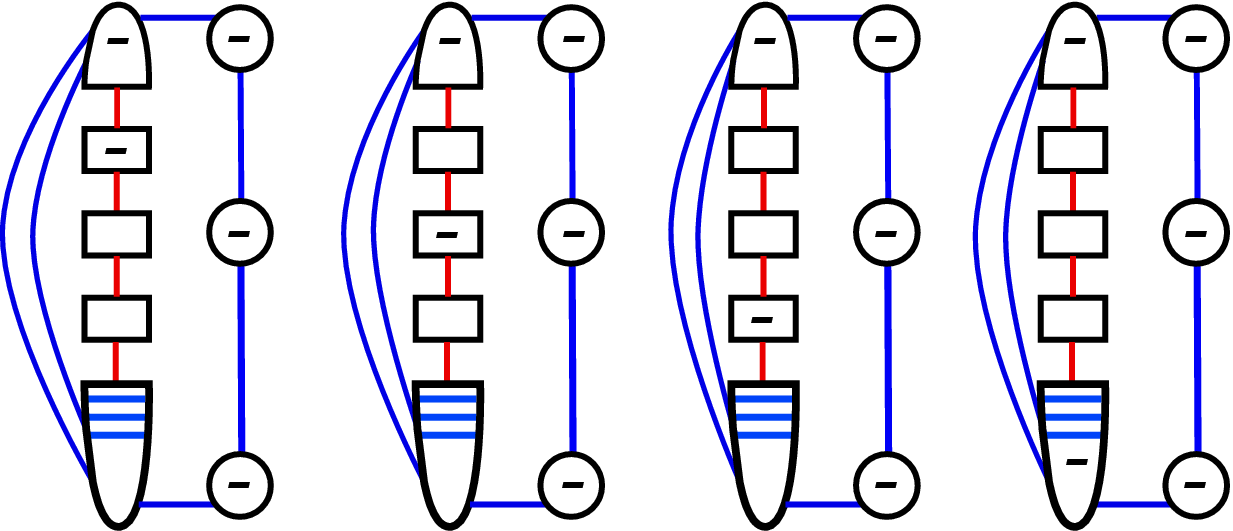}
			\qquad\qquad
			\includegraphics[scale=0.25]{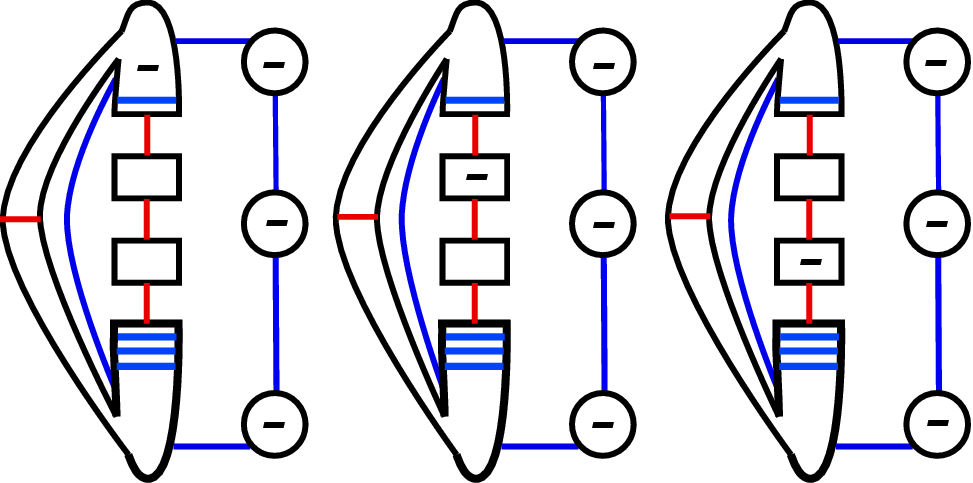}
		\end{center}
		\caption{Enhanced states of $B$, for $g\!=\!2$, $h\!=\!7$, $r\!=\!3$ and three external circles}
		\label{GeneratorsBWithExternalCircles}
	\end{figure}
	
	Consider the projection $\pi_{\cal B}:C(D)\rightarrow {\cal B}$. As in the proof of Theorem~\ref{BasicTheorem}, the map $\varepsilon \circ \pi_{\cal B}\circ d:C(D)\rightarrow R$ will be useful to prove that $V$ is not exact. Note first that, since $\pi_{\cal B}(V) = s^{0-}_{1;2, \dots, r+1}$, it follows that $\varepsilon(\pi_{\cal B}(V))=1$. We will then deduce that $V$ is not exact by just proving that the composition
	$$
	C(D) \stackrel{d}{\longrightarrow} C(D) \stackrel{\pi_{\cal B}}{\longrightarrow} {\cal B} \stackrel{\epsilon}{\longrightarrow} \mathbb{Z}
	$$
	has its image contained in $2\mathbb{Z}$, being enough to check that $\epsilon(\pi_{\cal B}(d(Y)) \in 2\mathbb{Z}$ for any enhanced state $Y$.
	
	Consider the matrix of the map $d:C^{r,2r-|s_AD|}(D) \rightarrow C^{r+1,2r-|s_AD|}(D)$, whose elements are $0$ or $\pm 1$. We have to check that, for each column, the number of non-zero entries in the rows that correspond to the elements of $B$ 
	is even. 
		
	A simpler idea is start with an element of $B$, integrate it, and consider the columns corresponding to any of the enhanced states that appear. 
	Since the states in $B$ do not have red bichords, these elements are the same as in the proof of Theorem~\ref{BasicTheorem}, Figure~\ref{GeneratorsBPrimitivas}, but with extra external circles, all of them with signs $-$. The proof is then completed by following the same argument there.
\end{proof}	

\begin{example} Figures~\ref{FigureXelementFORgDOShDOSrUNOCirculosExternosTRES} and \ref{FigureVelementFORgDOShDOSrUNOCirculosExternosTRES} show respectively the complete pictorial descriptions of $X$ and $V$ when $g=h=2$, $r=1$ and there are three external circles. 
	\begin{figure}[ht!] 
		\labellist
		\pinlabel {$+$} at 22 112.5
		\pinlabel {$s_{\,;1;\,}^{0}$} at 5 78
		\pinlabel {$+$} at 130 88
		\pinlabel {$s_{\,;1;\,}^{r}$} at 110 78
		\pinlabel {$+$} at 264 80
		\pinlabel {$s_{\,;1;\,}^{2}$} at 218 78
		\pinlabel {$+$} at 393.5 96.5
		\pinlabel {$s_{\,;1;\,}^{3}$} at 325 78
		\pinlabel {$+$} at 481 114
		\pinlabel {$s_{\,;1;\,}^{4}$} at 433 78
		\pinlabel {$+$} at 22 35.5
		\pinlabel {$s_{\,;2;\,}^{0}$} at 5 0
		\pinlabel {$+$} at 130 11
		\pinlabel {$s_{\,;2;\,}^{r}$} at 110 0
		\pinlabel {$+$} at 264 9.5
		\pinlabel {$s_{\,;2;\,}^{2}$} at 218 0
		\pinlabel {$+$} at 393.5 27
		\pinlabel {$s_{\,;2;\,}^{3}$} at 325 0
		\pinlabel {$+$} at 481 43.5
		\pinlabel {$s_{\,;2;\,}^{4}$} at 433 0
		\endlabellist
		\begin{center}
			\includegraphics[scale=0.7]{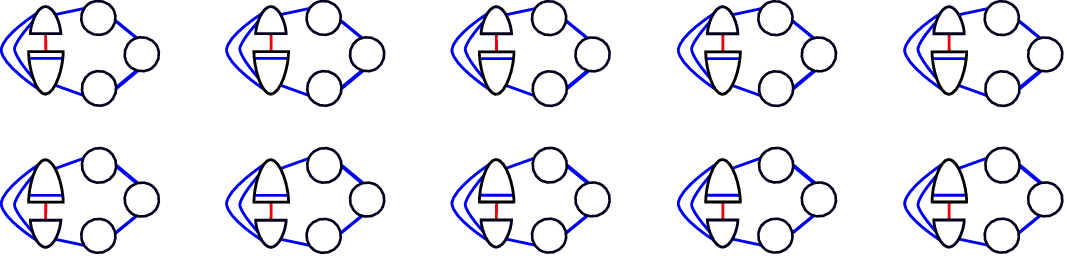}
		\end{center}
		\caption{Element $X = s_{\,;1;\,}^{0} + s_{\,;1;\,}^{r} - s_{\,;1;\,}^{2} + s_{\,;1;\,}^{3} - s_{\,;1;\,}^{4} + s_{\,;2;\,}^{0} + s_{\,;2;\,}^{r} - s_{\,;2;\,}^{2} + s_{\,;2;\,}^{3} - s_{\,;2;\,}^{4}$ for $g=h=2$, $r=1$ and three external circles}
		\label{FigureXelementFORgDOShDOSrUNOCirculosExternosTRES}
	\end{figure}
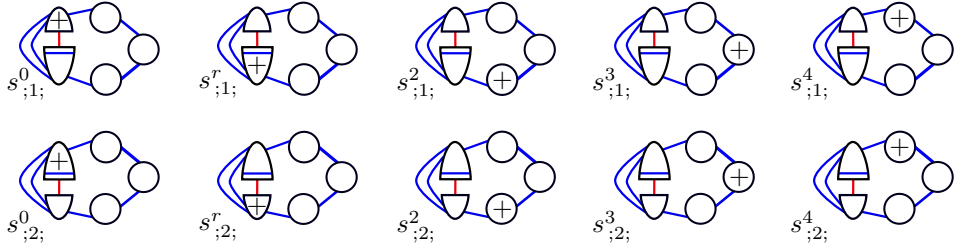
	
	\begin{figure}[ht!] 
		\labellist
		\pinlabel {$s_{1;1;\,}^{0-}$} at 5 0
		\pinlabel {$s_{1;2;\,}^{0-}$} at 110 0
		\pinlabel {$s_{2;1;\,}^{0-}$} at 225 0
		\pinlabel {$s_{2;2;\,}^{0-}$} at 327 0
		\endlabellist
		\begin{center}
			\includegraphics[scale=0.7]{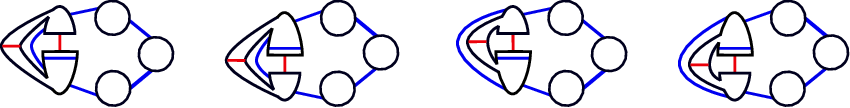}
		\end{center}
		\caption{Element $V = s_{1;1;\,}^{0-} + s_{1;2;\,}^{0-} + s_{2;1;\,}^{0-} + s_{2;2;\,}^{0-} $ for $g=h=2$, $r=1$ and three external circles. All signs are $-$ since there is no $H$-circles because $r=1$}
		\label{FigureVelementFORgDOShDOSrUNOCirculosExternosTRES}
	\end{figure}
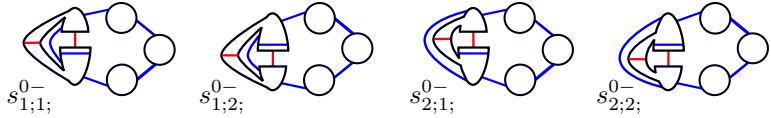
\end{example}

\begin{example}
	Corollary~\ref{MainCorollary} detects torsion of order two in the Whitehead link. Figure~\ref{FigureWhitehead} shows a diagram $D$ of this link that satisfies the hypothesis of Corollary~\ref{MainCorollary}, as its corresponding $A$-smoothing allows to see. 
	\begin{figure}[ht!] 
		\labellist
		\pinlabel {$D$} at 0 20
		\pinlabel {$s_AD$} at 330 20
		\endlabellist
		\begin{center}
			\includegraphics[scale=0.3]{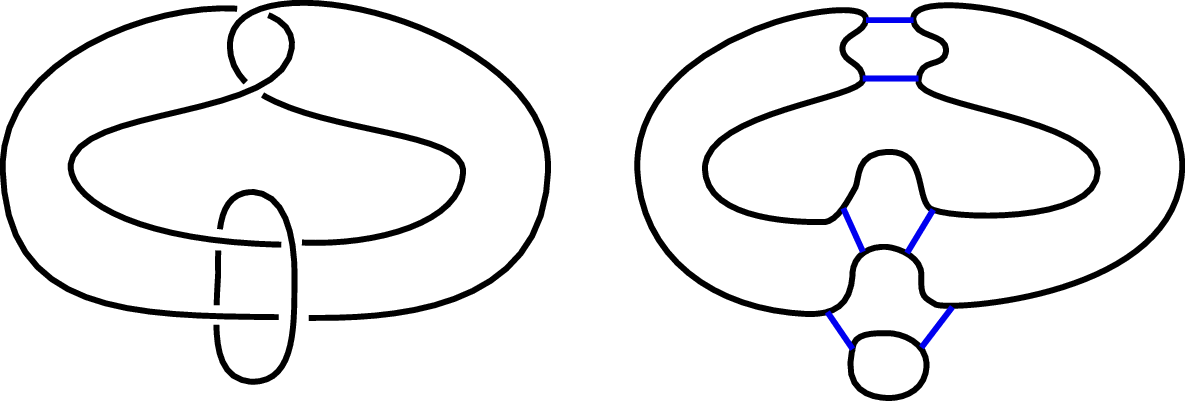}
		\end{center}
		\caption{A diagram $D$ of the Whitehead link, and the corresponding~$s_AD$}
		\label{FigureWhitehead}
	\end{figure}
\end{example}

\begin{example}
	Corollary~\ref{MainCorollary} detects torsion of order two in the  Borromean rings, the link shown on the left in Figure~\ref{FigureBorromeanRings}. In the same figure  another diagram $D$ of this link is shown, which satisfies the hypothesis of Corollary~\ref{MainCorollary}, as its corresponding $A$-smoothing allows to see. 
	\begin{figure}[ht!] 
		\labellist
		\pinlabel {$D$} at 70 5
		\pinlabel {$s_AD$} at 150 5
		\endlabellist
		\begin{center}
			\includegraphics[scale=1.3]{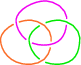}
			\qquad \qquad
			\includegraphics[scale=0.8]{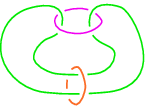}
			\qquad \qquad
			\includegraphics[scale=0.7]{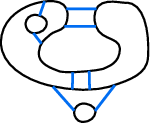}
		\end{center}
		\caption{Borromean rings, a {\it good} diagram $D$ for it and the corresponding~$s_AD$}
		\label{FigureBorromeanRings}
	\end{figure}
\end{example}

\begin{example} \label{ExampleNonAlternating}
	We now show a non-alternating knot that has torsion because of our torsion pattern (note that the construction we do can easily produce many more different examples). We start with a diagram $D_0$ of the non-alternating knot $8_{19}$ and observe that $s_AD_0$ has a bipartite graph without monochords (see Figure~\ref{FigureOchoDiecinueve}).
	\begin{figure}[ht!]
		\labellist
		\endlabellist
		\begin{center}
			\includegraphics[scale=0.3]{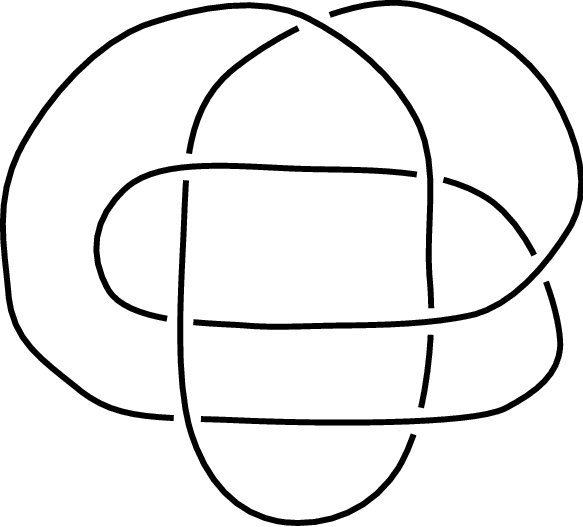}
			\qquad	
			\includegraphics[scale=0.3]{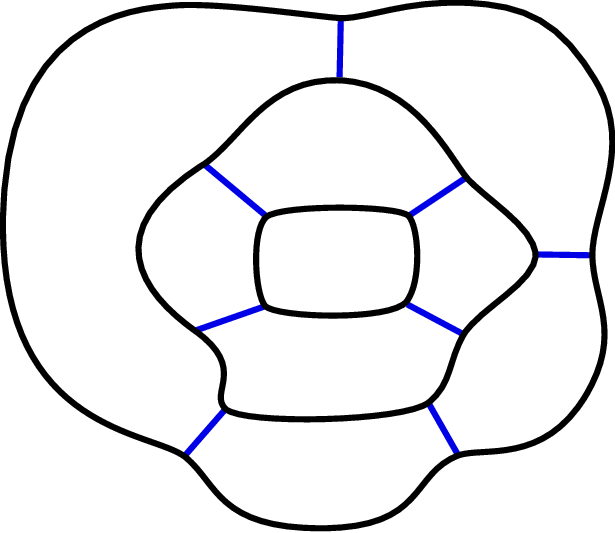}
		\end{center}
		\caption{Diagram $D_0$ of $8_{19}$ and $s_AD_0$.}
		\label{FigureOchoDiecinueve}
	\end{figure}
	
	We then insert $g+h$ monochords into a circle of $s_AD_0$, obtaining the $A$-smoothing $s_AD$ of a diagram $D$ (see Figure~\ref{FigureOchoDiecinueveConPegote}). The diagram $D$ satisfies the hypothesis of Corollary~\ref{MainCorollary}: it is a diagram of type $D(g,h)$ and the graph defined by $s_AD$ without the monochords is bipartite (since $s_AD_0$ was). Note that $D$ represents a knot (a link with one component) if $g$ and $h$ are both even.
	\begin{figure}[ht!]
		\labellist
		 \pinlabel {$g$} at 65 141
		 \pinlabel {$\dots$} at 65 129
		 \pinlabel {$h$} at 163 138
		 \pinlabel {$\vdots$} at 150 145
	     \pinlabel {$g$} at 470 150
		 \pinlabel {$\dots$} at 470 138
		 \pinlabel {$h$} at 547 127
		 \pinlabel {$\vdots$} at 537 132		
		 \pinlabel {$s_AD$} at 50 20
		 \pinlabel {$D$} at 390 20
		\endlabellist
		\begin{center}
			\includegraphics[scale=0.5]{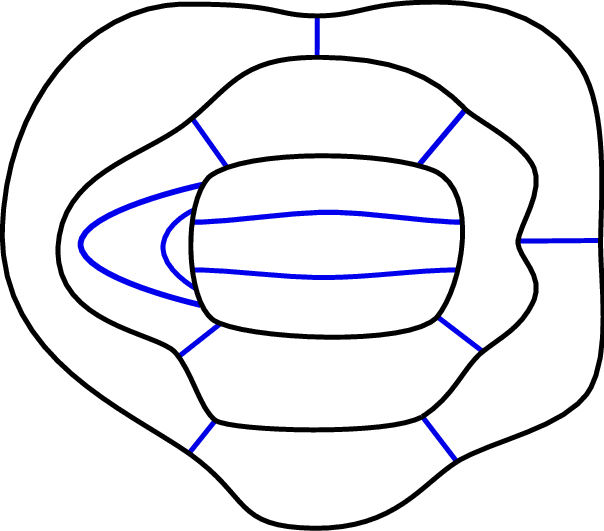}
			\qquad
			\includegraphics[scale=0.5]{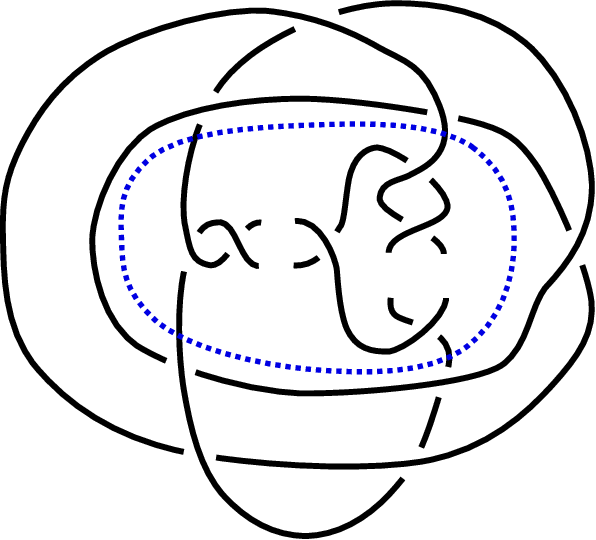}
		\end{center}
		\caption{From $s_AD_0$ we construct the convenient $s_AD$. On the right, the corresponding diagram $D$}
		\label{FigureOchoDiecinueveConPegote}
	\end{figure}
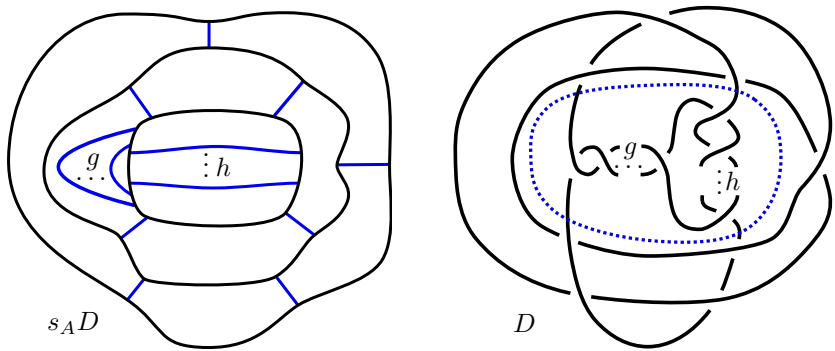
	
	 For example, if $g=h=2$, KnotInfo \cite{KnotInfo}, via KnotFinder, tells us that this knot is $K=11n_{61}$ or its mirror image. Since $D$ has $p=8$ positive crossings, $n=4$ negative crossings and $|s_AD|=3$, by Corollary~\ref{MainCorollary} there is a torsion element of order two in  $Kh^{2-n,2-|s_AD|+p-2n}(K) = Kh^{-2,-1}(K)$. In fact $Kh^{-2,-1}(11n_{61}\!\,^*) = \mathbb{Z} \oplus \mathbb{Z}_2$.
	
%
\end{example}

We end with two remarks.
\begin{remark}
Theorem~\ref{BasicTheorem} and Corollary~\ref{MainCorollary} are no longer true when $r$ is even, as Example~\ref{ExampleMirrorSixOne} shows (the column for $i=3$, corresponding to $r=2$, has no torsion). However, some generalizations of these results are possible; 
even so, the elements that produce torsion are necessarily different and the combinatorics involved is more convoluted. This is work of a forthcoming paper.  
\end{remark}

\begin{remark}
	In \cite{Kindred}, the non-zero elements in the Khovanov homology are presented in connection with the plumbing construction. Related to the torsion problem, this suggests the following general question: let $S_0\ast S'$ be a plumbing of two surfaces $S_0$ and $S'$. Is there any relation between the torsion elements of $S_0\ast S'$ and those of the factors $S_0$ and $S'$? Notice that Corollary~\ref{MainCorollary} provides examples on this line. To be more specific, in Figure~\ref{FigurePlumbing} we construct a plumbing of two surfaces, $S_0$ and $S'$. The boundary of $S_0$ is the pretzel link $P(-1,-1,-1, 3)$; on the other hand, $S'$ is an example of primitive flat surface. The boundary of the plumbing, $K=\partial (S_0\ast S')$, is the knot for which we obtain torsion elements according to Corollary~\ref{MainCorollary}, for $g=h=3$ and three external circles.  
	\begin{figure}[ht!]
	\labellist
		\pinlabel {$S_0$} at 10 20
		\pinlabel {$S'$} at 140 20
		\pinlabel {$S_0\ast S'$} at 390 20
	\endlabellist
	\begin{center}
		\includegraphics[scale=0.6]{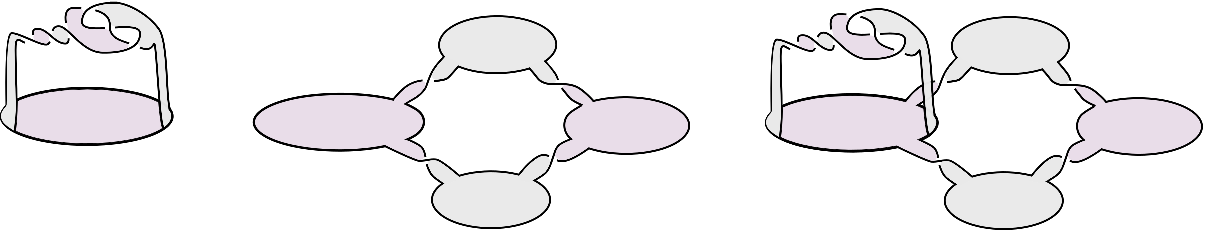}
	\end{center}
	\caption{Surfaces $S_0$ and $S'$ and the plumbing $S_0\ast S'$}
	\label{FigurePlumbing}
	\end{figure}
\end{remark}

\begin{center}
\begin{tabular}{c}
Raquel D\'{\i}az\\
Department of Algebra, Geometry and Topology \\
Facultad de Ciencias Matem\'aticas, Universidad Complutense de Madrid \\
 {\it radiaz@ucm.es} 
\end{tabular}

\begin{tabular}{c}
Pedro M. G. Manch\'on\\
Department of Applied Mathematics to Industrial Engineering \\
ETSIDI, Universidad Polit\'ecnica de Madrid \\
{\it pedro.gmanchon@upm.es} \\
\end{tabular}
\end{center}

\end{document}